\documentclass{article}
\usepackage[utf8]{inputenc}
\usepackage{amsfonts,amsmath}
\usepackage{mdframed}
\usepackage{graphicx}
\usepackage[
    backend=biber,
    style=alphabetic,
    maxbibnames=99,    
    minbibnames=99
]{biblatex} 
\usepackage{hyperref}
\usepackage{algpseudocode}
\usepackage{multirow}
\usepackage{array}
\usepackage{subcaption}

\usepackage{tikz}
\usetikzlibrary{arrows.meta,,positioning}

\usepackage{caption}

\usepackage[ruled,boxed,linesnumbered,commentsnumbered,longend]{algorithm2e}

\usepackage{comment}

\usepackage{amsthm}

\usepackage{multirow}

\newtheorem{lemma}{Lemma}

\newtheorem{proposition}{Proposition}[section]

\newtheorem{assumption}{Assumption}[section]
\newtheorem{example}{Example}[section]

\title{Dynamic Programming for\\ A Linear-Gaussian Dual Control Problem\footnote{Y. Wang and H. Hjalmarsson are with the Department of Decision and Control Systems, KTH Royal Institute of Technology, Sweden. Y. Li is with the Fulton School of Engineering, Arizona State University, AZ, USA.}}

\author{Ying Wang, Yuchao Li, and H{\aa}kan Hjalmarsson}

\usepackage{setspace}
\begin{document}
\maketitle

\begin{abstract}
We consider an infinite-horizon discounted dynamic programming (DP) problem for a countable-innovation belief-state model. The model is a surrogate for a linear-Gaussian dual control problem involving an unknown parameter vector and Gaussian observation noise. The surrogate preserves the first two moments of the Gaussian innovation and the resulting posterior second-moment structure. We allow unbounded belief and control spaces and nonnegative stage costs with quadratic growth in the posterior mean and covariance. Using a posterior-moment weight, we show that the Bellman operator is a contraction in a weighted sup-norm space. We characterize the optimal cost as the unique solution of the Bellman equation in this space, establish the existence of optimal stationary policies. We further prove weighted sup-norm convergence of value iteration, policy iteration, and optimistic policy iteration under vanishing approximation errors. The convergence results under vanishing approximation errors rely on general weighted sup-norm contraction properties and are applicable beyond the particular dual control problem considered here.
\end{abstract}


\vspace{1pc}

\section{Introduction}
\label{sec:intro}

A basic feature of dual control is that an input affects both current performance and the information available for future decisions, and hence future costs~\cite{feldbaum1963dual,kumar1985survey,KULCSAR1996135,LARSSON20161,mesbah2018stochastic,derpich2023dual,meijer2026dual,BALTUSSEN2026101577}. We study this effect for the linear-Gaussian regression model
\begin{equation}
    y_k=u_k'\theta+e_k,
\label{eq:io}
\end{equation}
where prime denotes transposition, $\theta$ is an unknown $n$-dimensional parameter, $u_k$ is the input, $y_k$ is the scalar observation, and $e_k$ is Gaussian noise. Under a Gaussian prior on $\theta$, the posterior distribution remains Gaussian with mean $m_k$ and covariance $P_k$. We therefore take the pair $(m_k,P_k)$ as the belief state $x_k$. The input $u_k$ incurs a current cost and affects the evolution of the belief state through Eq.~\eqref{eq:io}. Our objective is to develop a tractable discounted infinite-horizon dynamic programming (DP) theory motivated by this linear-Gaussian dual control problem, while accommodating unbounded belief and input spaces and stage costs with quadratic growth.

This problem has a long history in the control and economics literature. In particular, Prescott~\cite{prescott1972multi} studied this regression problem and emphasized the tradeoff between current performance and experimentation. Kiefer and Nyarko subsequently developed an exact discounted infinite-horizon DP theory \cite{easley1988controlling,kiefer1989optimala}. Under the assumptions of a compact input space and bounded continuous costs, they established continuity of the optimal cost function, the Bellman optimality equation, and the existence of stationary optimal policies by a contraction argument on the space of bounded continuous functions. However, their theory does not cover the unbounded input spaces and quadratically growing costs considered here. Subsequent work has focused primarily on computational or approximate approaches, including numerical value iteration~\cite{kiefer1989value,wieland2000learning,wieland2000monetary}, rollout~\cite{han2006approximate}, convex-optimization-based approximations~\cite{lobo1999policies,wang2024regret}, and receding-horizon approximations~\cite{wang2025revisiting}. Related linear-Gaussian observation models have also been studied in the Bayesian linear-bandit literature, with an emphasis on regret analysis; see, e.g., \cite{zhu2026prior}.

The first difficulty in extending the infinite-horizon theory in \cite{easley1988controlling,kiefer1989optimala} to the present problem is that the relevant cost functions are generally unbounded. As a result, the usual contraction analysis on bounded functions is not directly applicable. A natural alternative is weighted sup-norm methods~\cite{lippman1975dynamic,wessels1977markov,vannunen1978DP,bertsekas2012weighted,bertsekas2022abstract,ma2022unbounded,toda2024unbounded}. In particular, the analyses in \cite{ma2022unbounded,toda2024unbounded} rely on an inequality of the form
$$E\{v(x_{k+1})\mid x_k=x,u_k=u\}\leq c v(x),$$
for some constant $c>0$, where $v(x)$ is a weight function used to define the weighted sup-norm space.

For the quadratic-growth costs considered here, we introduce the posterior moment weight
\begin{equation}
    v(x)=1+\|m\|^2+\operatorname{tr}(P)
\label{eq:weight}
\end{equation}
for the state $x=(m,P)$, where $\|\cdot\|$ denotes the Euclidean norm and $\operatorname{tr}(\cdot)$ denotes the trace function. In contrast to the inequality used in \cite{ma2022unbounded,toda2024unbounded}, a key observation of this paper is that this particular weight satisfies an exact preservation identity under the belief-state transition. Specifically, we prove that
\begin{equation}
    E\left\{v(x_{k+1})\mid x_k=x,\,u_k=u\right\}=v(x)
\label{eq:weight_preservation}
\end{equation}
for every state-input pair $(x,u)$, as will be shown in Section~\ref{sec:weight}. This exact identity is the main structural property underlying the contraction and algorithmic analysis developed in this paper.

A second and logically separate difficulty concerns measurability. For the original linear-Gaussian problem, the posterior-mean update is driven by a continuous Gaussian innovation. With continuous innovations and an uncountable input space, one must ensure that the conditional expectations and minimizations appearing in the DP analysis and algorithms preserve the appropriate measurability properties. In particular, even when a function is Borel measurable in the state and input, its infimum over the control need not be Borel measurable as a function of the state; see \cite{bertsekas1978stochastic,shreve1978alternative}.

These issues can be treated within the general theory of stochastic control on Borel spaces, including the universal-measurability framework of Bertsekas and Shreve~\cite{bertsekas1978stochastic}. Under additional continuity and compactness-type conditions, Feinberg et al.~\cite{feinberg2012average,feinberg2016partially} develop results for Markov decision processes and partially observed Markov decision processes (POMDPs) with Borel state and control spaces, allowing noncompact control sets and unbounded stage costs. Their theory establishes the Bellman equation, lower semicontinuity of the optimal cost, existence of stationary optimal policies, and pointwise convergence of value iteration (VI for short). More recent work further develops verifiable continuity conditions for broad classes of partially observed models and nonlinear filters, including Kalman-type filters~\cite{feinberg2024sufficient,feinberg2025continuity}. Thus, the continuous-innovation model may in principle be treated within this more general Borel-space theory.

Our objective here is different. We seek fixed-point and algorithmic conclusions that exploit the particular posterior-moment structure of the regression model. The general Borel-space results described above do not directly provide a contraction of the Bellman operator in a function space accommodating the quadratic growth considered here, or uniqueness of the Bellman fixed point in such a weighted space. Nor do they directly establish weighted sup-norm convergence guarantees for the DP algorithms considered below under vanishing approximation errors.

To isolate this contractive structure without introducing the additional measurable function and measurable selection machinery required in the general Borel-space theory, we replace the continuous Gaussian innovation by a countable-support innovation with the same first two moments, namely zero mean and unit variance. The resulting model is a surrogate for the original linear-Gaussian problem rather than an equivalent representation. This choice preserves the posterior second-moment identity underlying the weighted-space analysis, while reducing the conditional expectations in the Bellman operators to countable sums. It also has a natural computational interpretation, since numerical treatment of the original Gaussian expectation generally requires some form of discrete approximation. Quantifying the relation between the optimal costs and policies of the countable- and continuous-innovation models is left for future research.

A further issue arises for policy-based algorithms. Although the Bellman operator is contractive on the weighted space, a policy-specific Bellman operator need not map this space into itself for an arbitrary policy. We therefore identify a class of policies, which we call \emph{proper} policies, for which the policy-specific operators are well posed and contractive, and show that policy iteration (PI) and optimistic PI generate policies that remain within this class. This makes it possible to establish convergence even when the policy-improvement steps are performed with vanishing approximation errors.

Within this countable-innovation model, we develop a discounted infinite-horizon DP theory for unbounded belief and input spaces and nonnegative stage costs with quadratic growth. Specifically, our contributions are as follows.
\begin{itemize}

\item We introduce the posterior-moment weight in Eq.~\eqref{eq:weight} and prove the weight-preservation identity~\eqref{eq:weight_preservation}. Using this identity, we establish self-map and contraction properties of the Bellman operator and of the policy-specific Bellman operators associated with proper policies on the corresponding weighted sup-norm space.

\item We characterize the optimal cost as the unique solution of the Bellman equation in the weighted space, establish its lower semicontinuity, and prove the existence of optimal stationary policies. We also show that membership of the cost function of a policy in the weighted space is equivalent to properness of the policy, and that every optimal stationary policy is proper.

\item We establish weighted sup-norm convergence of VI, PI, and optimistic PI under vanishing approximation errors. For PI and optimistic PI, we further show that the generated policies remain proper; for optimistic PI, both the iterated cost functions and the cost functions of the generated policies converge to the optimal cost. These vanishing-error convergence results are not specific to the present belief-state model and extend to general weighted sup-norm contractive DP models under suitable assumptions.

\end{itemize}

The remainder of the paper is organized as follows. Section~\ref{sec:prb_set} formulates the countable-innovation belief-state DP problem. Section~\ref{sec:weight} introduces the weighted sup-norm space, proves the posterior weight-preservation property, and establishes the self-map and contraction properties of the Bellman operators. Section~\ref{sec:bellman} connects the Bellman fixed point with the optimal cost function and proves the existence and properness of optimal stationary policies. Section~\ref{sec:vi_algo} studies VI, PI, and optimistic PI. Section~\ref{sec:numerical_result} presents a numerical example, and Section~\ref{sec:conclusion} concludes the paper.

\section{Problem Formulation}
\label{sec:prb_set}

In this section, we describe the Bayesian update associated with the linear-Gaussian model introduced in Section~\ref{sec:intro}. The posterior distribution of the unknown parameter is characterized by its mean and covariance, which we use as the state of the DP problem. For the DP analysis, we replace the Gaussian innovation by a countable-support random variable with the same first two moments. We then introduce the stage cost, the infinite-horizon discounted criterion, and the growth conditions used in the subsequent analysis.

\subsection{Belief State Update for Linear-Gaussian Model}
\label{sec:posterior_update}

We consider the linear-Gaussian model~\eqref{eq:io}, where $\theta\in\Re^n$ is an unknown parameter, $y_k\in\Re$ is the measured output, and ${e_k}$ is a sequence of independent Gaussian random variables with zero mean and variance $\sigma^2$. The input $u_k\in U$ is the control variable, where $U\subset\Re^n$ is a nonempty control space.

We model the unknown parameter $\theta$ as a Gaussian random vector with mean $m_0\in\Re^n$ and covariance $P_0\in S_{+}^n$, where $S_{+}^n$ denotes the set of $n$ by $n$ positive semidefinite matrices. The initial information is denoted by $I_0=(m_0,P_0)$. For $k\geq 1$, let $I_k$ denote the information available prior to the selection of $u_k$:
$$I_k=(m_0,P_0,u_0,y_0,u_1,y_1,\ldots,u_{k-1},y_{k-1}).$$
Thus, $u_k$ is selected on the basis of $I_k$.

Under the linear-Gaussian model~\eqref{eq:io}, the posterior distribution of $\theta$ given $I_k$ remains Gaussian; see, cf.~\cite[Sec.~11]{ljung1999system}. We denote its mean and covariance by $m_k\in\Re^n$ and $P_k\in S_{+}^n$, respectively. The corresponding Bayesian update can be written in innovation form as
\begin{equation}\label{eq:m_update}
    m_{k+1}= m_k+\frac{P_ku_k}{\sqrt{\sigma^2+u_k'P_ku_k}}\epsilon_k,    
\end{equation}
and
\begin{equation}\label{eq:P_update}
    P_{k+1}=P_k-\frac{P_ku_ku_k'P_k}{\sigma^2+u_k'P_ku_k}.
\end{equation}
where $\epsilon_k$ is the standardized innovation and has zero mean and unit variance. A derivation is given in Appendix~\ref{app:innovation} for completeness.

As discussed in Section~\ref{sec:intro}, our DP analysis replaces the continuous Gaussian innovation by a countable-support innovation with the same first two moments, namely zero mean and unit variance. These are precisely the innovation properties used in the posterior-moment analysis developed later.

\subsection{Dynamic Programming Formulation}\label{sec:dp_prob_formulation}
The posterior update in Section~\ref{sec:posterior_update} provides the state equation for the DP formulation. Given the posterior mean and covariance $(m_k,P_k)$ and the control $u_k$, the distribution of $(m_{k+1},P_{k+1})$ is determined independently of the preceding observations and controls. Thus, $x_k=(m_k,P_k)$ is an information state, and the posterior recursion defines a controlled Markov system.

For the DP analysis, we use the posterior recursion derived in Section~\ref{sec:posterior_update} with the Gaussian innovation replaced by a countable-support random variable $w_k$ with zero mean and unit variance.
We thus consider a discounted infinite-horizon DP problem with state space $X=\Re^n\times S_{+}^n$. At state $x_k=(m_k,P_k)$, a control $u_k\in U$ is selected, and the next state is given by
\begin{equation}
\label{eq:dynamics_def}
    m_{k+1}=f_1(x_k,u_k,w_k),
    \qquad
    P_{k+1}=f_2(P_k,u_k),
\end{equation}
The state equations $f_1$ and $f_2$ have the same form as the Gaussian posterior update Eqs.~\eqref{eq:m_update}-\eqref{eq:P_update}, except that the Gaussian innovation is replaced by the countable-support innovation $w_k$. 

For notational convenience, we write the state equation as
\begin{equation}
\label{eq:dynamics_f_def}
x_{k+1}=f(x_k,u_k,w_k),
\qquad k=0,1,\ldots,
\end{equation}
where $X\times U\times W\mapsto X$ is defined by Eq.~\eqref{eq:dynamics_def}, and $W$ denotes the countable support of the innovation $w_k$. Associated with a state-control pair $(x_k,u_k)$ is a deterministic nonnegative stage cost $g(x_k,u_k)$ defined as
\begin{equation}
    \label{eq:cost_def}
    g(x_k,u_k)
    =
    \big(u_k-\nu(x_k)\big)'R\big(u_k-\nu(x_k)\big)
    +
    \psi(x_k),
\end{equation}
where $R$ is a positive definite matrix, $\nu: X\mapsto\Re^n$ is a given function, referred to as the reference control law, and $\psi: X\mapsto[0,\infty)$ is a belief-state penalty. Thus, the first term penalizes deviation of the control from $\nu(x)$, while $\psi(x)$ penalizes the belief state. This cost structure is motivated by and generalizes the practical quadratic model discussed in Appendix~\ref{app:cost_motivation}.

We are interested in feedback policies of the form $\pi=\{\mu_0, \mu_1, \dots\}$, where each $\mu_k$ is a function mapping $X$ to $U$. We denote by $\mathcal M$ the set of all functions that maps $X$ to $U$, and by $\Pi$ the set of all policies. For a given policy $\pi$, its cost function starting from an initial state $x_0$ is defined as\footnote{The limit in Eq.~\eqref{eq:cost_f_2i} is well-defined in view of nonnegativity of $g(x_k,u_k)$; cf. Eq.~\eqref{eq:cost_def}.}
\begin{equation}
    \label{eq:cost_f_2i}
    J_\pi(x_0)=\lim_{N\to\infty}\mathop{E}
    \left\{\sum_{k=0}^{N-1}\alpha^kg\big(x_k,\mu_k(x_k)\big)\right\},
\end{equation}
subject to the state equation constraint \eqref{eq:dynamics_def}. Here, the expectation is taken with respect to the disturbance sequence $\{w_k\}$, and the scalar $\alpha\in (0,1)$ is a \emph{discount factor}. If a policy $\pi$ is of the form $\pi=\{\mu,\mu,\dots\}$, we call it \emph{stationary}, and when no confusion arises, we denote the stationary policy by $\mu$ and its cost function by $J_\mu$. The optimal cost function is defined as
\begin{equation}
    \label{eq:j_star_def}
    J^*(x)=\inf_{\pi\in \Pi} J_\pi(x),\qquad \forall x\in X,
\end{equation}
and a policy $\pi^*$ is said to be optimal, if it attains the minimum in Eq.~\eqref{eq:j_star_def} for all $x$, that is
$$J_{\pi^*}(x)=J^*(x)\qquad \forall x\in X.$$

Although $\nu(x)$ minimizes the one-stage cost at state $x$, it need not be optimal for the infinite-horizon problem, since the current control affects the belief-state transition and, consequently, future costs. The stage cost~\eqref{eq:cost_def} may also be unbounded as the belief state grows. Motivated by the quadratic model considered in Appendix~\ref{app:cost_motivation}, we impose the following growth conditions on the reference control $\nu$ and the belief-state penalty $\psi$. These conditions capture the growth properties of that model while allowing the analysis to be developed in a more general form.

\begin{assumption}
\label{ass:prb_set}
    There exist constants $c_\nu>0$, $c_\psi>0$ such that
    $$\|\nu(x)\|^2\le c_\nu v(x), \qquad \psi(x) \leq c_{\psi} v(x) ,\qquad \forall x\in X.$$
\end{assumption}

The following example illustrates how Assumption~\ref{ass:prb_set} arises naturally from a quadratic cost model of practical interest.
\begin{example}\label{eg:affine}
Suppose the reference control law $\nu$ and belief-state penalty $\psi$ in the stage cost \eqref{eq:cost_def} take the following form:
$$\nu(x)=Km+b,\qquad \psi(x)=m'Hm+\operatorname{tr}(GP),$$
where $K$ is a nonzero $n$ by $n$ matrix, $b\in \Re^n$, and $H$ and $G$ are positive definite matrices. This choice of $\nu$ and $\psi$ arises from the quadratic cost model derived in Appendix~\ref{app:eg}. A straightforward calculation shows that both conditions in Assumption~\ref{ass:prb_set} are satisfied.
\end{example}

\section{Weighted Sup-Norm Space, Bellman Operators, and Proper Policy}\label{sec:weight}
This section develops the weighted sup-norm framework used throughout the remainder of the paper. Because the stage cost and the resulting cost functions may be unbounded, we first introduce a weighted sup-norm space based on the posterior second moment. We then establish a weight-preservation property of the belief-state dynamics, which is used to prove self-map and contraction properties of the Bellman operators. Finally, we introduce the class of proper policies and establish the equivalence between properness and membership of the corresponding policy cost function in the space. These results provide the foundation for the optimality and algorithmic results developed in the subsequent sections. For several of the standard results in this section, we provide the main argument and its connection to the weight-preservation property, while detailed proofs are given in Appendix~\ref{app:additional_proof}.

\subsection{Functional Space and Weight Preservation Property}
The choice of the weight is guided by the growth conditions imposed on the reference control law $\nu(x)$ and the belief-state penalty $\psi(x)$ in Assumption~\ref{ass:prb_set}. These conditions show that the growth of both terms is bounded by the
function $v(x)$ defined in Eq.~\eqref{eq:weight}; i.e.,
\begin{equation}
    v(m,P)=1+\|m\|^2+\operatorname{tr}(P).
\label{eq:weight_repeat}
\end{equation}
Since $v(x)>0$ for all $x\in X$, we use $v$ as the weight function for the weighted sup-norm space.

We now define the corresponding weighted sup-norm space. Let $\mathcal B(X)$ be the set of real-valued functions $J:X\mapsto\Re$ such that $|J(x)|/v(x)$ is bounded over $X$, i.e.,
$$\mathcal{B}(X)=\left\{J:X\mapsto \Re \ \middle|\ \sup_{x\in X}\frac{|J(x)|}{v(x)}<\infty\right\}.$$
We equip $\mathcal{B}(X)$ with a norm $\|\cdot\|_v$ defined as
$$\|J\|_v=\sup_{x\in X}\frac{|J(x)|}{v(x)},\qquad \forall J\in\mathcal{B}(X).$$
The space $\mathcal{B}(X)$ is complete under this norm; see, e.g., \cite[Appendix~B.2]{bertsekas2022abstract}. We work primarily with the closed subset
\begin{equation}
    \label{eq:def_vspace}
    \mathcal{B}_+(X)=\left\{J\in \mathcal{B}(X) \ \middle|\ J(x)\geq 0,\ \forall x\in X\right\}.
\end{equation}
The cost functions considered in this paper will be shown to belong to this subset. 

We first establish a key property about the weight function $v$ in Eq.~\eqref{eq:weight_repeat} and the state equation $f$ in Eq.~\eqref{eq:dynamics_f_def}. Roughly speaking, it states that at state $x_k$, regardless of the control $u_k$ we apply, the expected value of $v(x_{k+1})$ equals $v(x_k)$. This weight-preservation property is a distinctive feature of the present belief-state dynamics and is central to the analysis that follows.

\begin{proposition}[Weight Preservation]
\label{prop:weight_preservation}
For every $(x,u)\in X\times U$, there holds
\begin{equation}
\label{eq:weight_pre}
    E\big\{v\big(f(x,u,w)\big)\big\}=v(x).
\end{equation}
Moreover, if $J\in\mathcal B_+(X)$, we have
\begin{equation}
        0\le
     E\big\{J\big(f(x,u,w)\big)\big\}
    \le
    \|J\|_v v(x).\label{eq:bounded_future_cost}
\end{equation}
\end{proposition}

\begin{proof}
In view of Eqs.~\eqref{eq:dynamics_f_def} and \eqref{eq:weight_repeat}, we have
$$v\big(f(x,u,w)\big)=1+\big\|f_1(x,u,w)\big\|^2+\operatorname{tr}\big(f_2(P,u)\big).$$
Taking expectation with respect to $w$ yields
$$E\big\{v\big(f(x,u,w)\big)\big\}=1+E\Big\{\big\|f_1(x,u,w)\big\|^2\Big\}+\operatorname{tr}\big(f_2(P,u)\big).$$
Recall that $w$ is a random variable with mean $0$ and variance $1$. Therefore we obtain
$$E\Big\{\big\|f_1(x,u,w)\big\|^2\Big\}=\|m\|^2+\operatorname{tr}\big(P-f_2(P,u)\big).$$
Combining the preceding two equations yields,
\begin{equation}\label{eq:weight_kept}
    \begin{aligned}
E\big\{v\big(f(x,u,w)\big)\big\}&=1+\|m\|^2+\operatorname{tr}\big(P-f_2(P,u)\big)+\operatorname{tr}\big(f_2(P,u)\big) \\
&=1+\|m\|^2+\operatorname{tr}(P)=v(x).
\end{aligned}
\end{equation}

For every $J\in\mathcal{B}_+(X)$, in view of the definition~\eqref{eq:def_vspace}, we have
$$0\leq J\big(f(x,u,w)\big)\le \|J\|_v\,v\big(f(x,u,w)\big),\qquad \forall x\in X,\,u\in U,\,w\in W.$$
Taking conditional expectation on both sides with respect to $w$, and using Eq.~\eqref{eq:weight_kept} yields the desired relation \eqref{eq:bounded_future_cost}.
\end{proof}

The weight-preservation identity has a simple Bayesian interpretation. Since $E\{\|\theta\|^2\,|\, x\}=\|m\|^2+\operatorname{tr}(P)$, the weight $v(x)$ is one plus the posterior second moment of $\theta$. Thus, Prop.~\ref{prop:weight_preservation} states that this posterior second moment is preserved in expectation under the belief-state update.

\subsection{Operator Properties on the Weighted Space}
We consider two Bellman operators. For a policy $\mu\in\mathcal M$, the \emph{$\mu$-Bellman operator} $T_\mu$ is defined by
\begin{equation}
    \label{eq:t_t_mu}
    (T_\mu J)(x)=g\big(x,\mu(x)\big)+\alpha E\Big\{J\Big(f\big(x,\mu(x),w\big)\Big)\Big\},\qquad \forall x\in X.
\end{equation}
The \emph{Bellman operator} $T$ is defined by
$$(TJ)(x)=\inf_{u\in U}\Big[g(x,u)+\alpha E\big\{J(f(x,u,w))\big\}\Big],
\qquad \forall x\in X,$$
where the expectation is finite by Prop.~\ref{prop:weight_preservation}. Since $\mathcal M$ is the set of functions that map $X$ to $U$, we have
$$(TJ)(x)=\inf_{\mu\in\mathcal M}(T_\mu J)(x),\qquad \forall J\in \mathcal{B}_+(X),\,x\in X.$$

An important property that holds for $T_\mu$ and $T$ is that for all $J,\tilde J\in \mathcal{B}_+(X)$, if $J(x)\leq \tilde J(x)$ for all $x\in X$, then
\begin{equation}
    \label{eq:monotone}
    (T_\mu J)(x)\leq (T_\mu \tilde J)(x), \quad (TJ)(x)\leq (T\tilde J)(x),\quad \forall x\in X,\mu\in \mathcal M.
\end{equation}
This property is referred to as the \emph{monotonicity} property of the operators $T_\mu$ and $T$, and will play an important role in our analysis. In our terminology, all equations and inequalities are meant to be pointwise. Accordingly, the monotonicity property can be stated succinctly as
$$J\leq \tilde J\ \ \text{implies}\ \ T_\mu J\leq T_\mu \tilde J,\ \ \forall \mu\in \mathcal M, \ \ TJ\leq T\tilde J.$$

To ensure that $T_\mu$ maps $\mathcal B_+(X)$ into itself, we require the stage cost generated by $\mu$ to have the same growth bound as the functions in $\mathcal B_+(X)$. In particular, we call a policy $\mu$ \emph{proper} if there exists a constant $c_{\mu}>0$ such that
$$g(x,\mu(x))\le c_\mu v(x),\qquad x\in X.$$
Equivalently, this means that the function $g\big(\cdot,\mu(\cdot)\big)$, when viewed as a function of $x$, belongs to $\mathcal{B}_+(X)$. We denote the set of proper policies by $\mathcal M_v$; i.e.,
$$\mathcal M_v= \left\{\mu\in \mathcal M\,|\,\mu \text{ is proper} \right\}.$$

In what follows, we show that $\mathcal M_v$ is nonempty.

\begin{proposition}[Existence of a Proper Policy]
\label{prop:proper}
Suppose Assumption~\ref{ass:prb_set} holds. There exists at least one proper policy, i.e., $\mathcal M_v\neq \emptyset$.
\end{proposition}

The detailed proof is given in Appendix~\ref{app:proper}, and we provide here a short explanation. Consider a constant policy $\mu(x)\equiv u_0$, where $u_0$ is some arbitrary element in $U$. By Assumption~\ref{ass:prb_set}, both $\|\nu(x)\|^2$ and $\psi(x)$ are bounded by constant multiples of $v(x)$. Since $u_0$ is fixed and $v(x)\geq 1$, the quadratic term $\bigl(u_0-\nu(x)\bigr)'R\bigl(u_0-\nu(x)\bigr)$ is also bounded by a constant multiple of $v(x)$. Hence for all $x\in X$, $g\bigl(x,\mu_c(x)\bigr)\leq c_{\mu_c}v(x)$ for some $c_{\mu_c}>0$, so that $\mu_c\in\mathcal M_v$.

The following two examples illustrate the nonemptiness of $\mathcal M_v$ for both bounded and unbounded control spaces.

\begin{example}
\label{eg:bounded_U}
Suppose that $U$ is bounded. Then every  policy $\mu:X\mapsto U$ is proper. Indeed, since $U$ is bounded, there exists $c_u>0$ such that $\|\mu(x)\|^2\le c_u$ for all $x\in X$. By Assumption~\ref{ass:prb_set}, repeating the proof in Prop.~\ref{prop:proper}, with the constant input $u_0$ replaced by any  policy $\mu(x)$, gives 
\begin{equation*}
    \begin{aligned}
    g(x,\mu(x))\le
    \Bigl(
    2\lambda c_u
    +
    2\lambda c_\nu
    +
    c_\psi
    \Bigr)v(x), \qquad \forall x\in X,
\end{aligned}
\end{equation*}
where $\lambda $ denotes the maximum eigenvalue of $R$. Hence $\mu\in\mathcal M_v$.
\end{example}

\begin{example}
\label{eg:unbounded_U}
Suppose that $U=\Re^n$. Then the reference control law $\nu(x)$ is admissible and proper. Indeed, by Assumption~\ref{ass:prb_set}, there exists $c_\psi$ such that,
\begin{equation*}
g(x, \nu(x)) =
    \bigl(\nu(x)-\nu(x)\bigr)'R\bigl(\nu(x)-\nu(x)\bigr)
    +
    \psi(x)= \psi(x) \le c_{\psi}v(x), \quad \forall x\in X.
\end{equation*}
Thus, $\nu\in\mathcal M_v$, and $\mathcal M_v\neq \emptyset$ holds.
\end{example}

\subsubsection{Self-Map Properties}
The property of properness provides a growth bound on the stage cost of a policy, based on which we can show that both $T_\mu$ and $T$ map $\mathcal{B}_+(X)$ to $\mathcal{B}_+(X)$. 

\begin{proposition}
\label{prop:self_map}
Let Assumption~\ref{ass:prb_set} hold.
\begin{itemize}
    \item[(a)] Suppose $\mu$ is a proper policy. Then its $\mu$-Bellman operator $T_\mu$ is a self-map on $\mathcal{B}_+(X)$, i.e., $T_{\mu}J \in \mathcal{B}_+(X)$ for all $J\in \mathcal{B}_+(X)$. 
    \item[(b)] The Bellman operator  $T$ is a self-map on $\mathcal{B}_+(X)$, i.e., $TJ \in \mathcal{B}_+(X)$ for all $J\in \mathcal{B}_+(X)$. 
\end{itemize}
\end{proposition}

Proof of this proposition is given in Appendix~\ref{app:self_map}, and we provide a brief explanation here. The main idea follows directly from the weight-preservation property in Prop.~\ref{prop:weight_preservation}. For any $J\in\mathcal B_+(X)$,
$$0\leq E\Big\{J\Big(f\big(x,u,w\big)\Big)\Big\} \leq
\|J\|_v v(x), \qquad \forall (x,u)\in X\times U.$$
For part~(a), if $\mu\in\mathcal M_v$, properness gives $g\big(x,\mu(x)\big)\leq c_\mu v(x).$ Hence both terms in $T_\mu J$ are nonnegative and bounded by constant multiples of $v(x)$, so that $T_\mu J\in\mathcal B_+(X)$.

For part~(b), Prop.~\ref{prop:proper} guarantees the existence of a proper policy $\mu\in\mathcal M_v$. Evaluating the expression defining $TJ$ at $u=\mu(x)$ gives an upper bound for $TJ$ of the same form, while nonnegativity of $g$ and $J$ gives $TJ\geq 0$. Thus $TJ\in\mathcal B_+(X)$.

In view of the self-map property, the operators $T_\mu$ and $T$ can be applied repeatedly. For $k\geq 0$, let $T_\mu^k$ and $T^k$ denote the result of $k$ fold compositions of $T_\mu$ and $T$, respectively. By convention, for all $J\in\mathcal B_+(X)$, we set
$T_\mu^0J=T^0J=J$.

\subsubsection{Contraction Properties}\label{sec:contraction}

Having established that the Bellman operators are self-maps on $\mathcal B_+(X)$, we next show that, they are contractions under the weighted sup-norm.

\begin{proposition}
\label{prop:contraction}
Suppose that  Assumption~\ref{ass:prb_set} holds.
\begin{itemize}
    \item[(a)] For every $\mu\in\mathcal M_v$, the $\mu$-Bellman operator $T_{\mu}:\mathcal{B}_+(X)\mapsto\mathcal{B}_+(X)$ satisfies 
\begin{equation*}
\|T_{\mu}J-T_{\mu}\tilde J\|_v\le \alpha \|J-\tilde J\|_v,
\quad \forall J,\tilde J\in\mathcal{B}_+(X).
\end{equation*}

    \item[(b)] The Bellman operator $T:\mathcal{B}_+(X)\mapsto\mathcal{B}_+(X)$ satisfies
\begin{equation*}
\|TJ-T\tilde J\|_v\le \alpha \|J-\tilde J\|_v,
\quad \forall J,\tilde J\in\mathcal{B}_+(X).
\end{equation*}
\end{itemize}
\end{proposition}

The contraction property also follows from Prop.~\ref{prop:weight_preservation}. For $J,\tilde J\in\mathcal B_+(X)$, 
$$|J(x)-\tilde J(x)|\leq\|J-\tilde J\|_vv(x),\qquad x\in X.$$
Hence, using $E\{v(f(x,u,w))\}=v(x)$, the difference between the expected future-cost terms is bounded by $\alpha\| J-\tilde J\|_v$. For $T_\mu$, the stage-cost terms cancel. For $T$, the same bound holds uniformly over $u\in U$, and therefore also after taking the infimum over $u$. Hence, both operators are contractions with modulus $\alpha$. Detailed proofs are given in Appendix~\ref{app:contraction}.

Since $\mathcal{B}_+(X)$ is complete under $\|\cdot\|_v$, Banach fixed-point theorem implies that $T$ has a unique fixed point, denoted by $\bar J$, satisfying $\bar J=T\bar J$. More details about Banach fixed-point theorem are given in Appendix~\ref{app:classical} Prop.~\ref{prop:banach}.

Let us also mention an important consequence of monotonicity and the contraction properties of the operators $T_\mu$ and $T$, which we will use frequently in our algorithmic analysis in Section~\ref{sec:vi_algo}. It is called the \emph{weight shift} property, which states that for all $J\in \mathcal{B}_+(X)$ and nonnegative scalar $c$, we have
\begin{equation}
    \label{eq:weight_shift}
    T_\mu(J+cv)\leq T_\mu J+\alpha cv,\quad\mu\in \mathcal M_v,\qquad T(J+cv)\leq T J+\alpha cv,
\end{equation}
where $(J+cv)(x)=J(x)+cv(x)$ for all $x$.\footnote{In fact, due to Eq.~\eqref{eq:weight_pre}, the stronger result holds
$$T_\mu(J+cv)=T_\mu J+\alpha cv,\qquad T(J+cv)=T J+\alpha cv,$$
where $\mu$ need not be proper; cf. Prop.~\ref{prop:weight_preservation}. For our purpose, however, the inequalities~\eqref{eq:weight_shift} suffice.}

\subsection{Equivalence Between Properness and Weighted Boundedness}
\label{subsec:proper_policy_weighted_value}
The preceding results use properness as a growth condition on the stage cost to ensure that the corresponding $\mu$-Bellman operator is well defined on $\mathcal B_+(X)$. The next proposition shows that properness is equivalent to the corresponding policy cost function $J_\mu$ belonging to $\mathcal B_+(X)$.

\begin{proposition}\label{prop:proper_value_equivalence}
A policy $\mu$ is proper, i.e., $\mu\in \mathcal M_v$, if and only if its cost function $J_\mu$ belongs to $\mathcal B_+(X)$.
\end{proposition}

\begin{proof}
We first show that if $J_{\mu}\in\mathcal B_+(X)$, then $\mu \in \mathcal M_v$. Since $J_\mu$ is the cost function of the policy $\mu$, it satisfies the policy evaluation equation:
$$J_{\mu}(x)=g\big(x,\mu(x)\big) +\alpha E\Big\{ J_{\mu}\big(f\big(x,\mu(x),w\big)\big)\Big\},\qquad \forall x\in X;$$
cf. Prop.~\ref{prop:classical}(a). Because $J_\mu\in\mathcal B_+(X)$, we have
$E\left\{J_\mu(f(x,\mu(x),w))\right\}\ge 0$. This implies that
$$ g(x,\mu(x))\le J_\mu(x)\le\|J_\mu\|_v v(x),\qquad \forall x\in X,$$
where the last inequality follows from $J_\mu\in\mathcal B_+(X)$. Thus we have $\mu\in\mathcal M_v$.

Conversely, by the definition of a proper policy, there exists a constant $c_{\mu}>0$ such that
$$g(x,\mu(x))\le c_{\mu}v(x), \qquad \forall x\in X.$$
Thus, its cost function is
\begin{equation*}
    \begin{aligned}
J_{\mu}(x_0)&=\lim_{N\to\infty}\mathop{E}\left\{\sum_{k=0}^{N-1}\alpha^kg\big(x_k,\mu(x_k)\big)\right\}\\
&\le \lim_{N\to\infty} \mathop{E}\left\{\sum_{k=0}^{N-1}\alpha^k c_{\mu}  v(x_k)\right\}
=\frac{c_{\mu}}{{1-\alpha}} v(x_0),
\end{aligned}
\end{equation*}
where the last equality comes from Prop.~\ref{prop:weight_preservation} and the tower property~\cite[Chapter 34]{billingsley1995probability}. Indeed, for a given $x_0$,
$$E\{v(x_{k+1})\mid x_0\}=E\big\{E\{v(x_{k+1})\mid x_k\}\mid x_0\big\}=E\{v(x_k)\mid x_0\}.$$
Since $E\{v(x_0)\mid x_0\}=v(x_0)$, induction gives $E\{v(x_k)\mid x_0\}=v(x_0)$, $k=0,1,\ldots$. Therefore we have $J_{\mu}\in\mathcal B_+(X)$.
\end{proof}

\section{Bellman's Equation and Optimal Policies}\label{sec:bellman}

The previous section established that the Bellman operator $T$ is a self-map and a contraction on the closed set $\mathcal B_+(X)$. We now show that the fixed point of the operator $T$ is the optimal cost function of the problem formulated in Section~\ref{sec:prb_set}. We first impose the following assumption.

\begin{assumption}\label{ass:close_U} The control space $U$ is closed. Furthermore,  the reference control law $\nu$ in the stage cost~\eqref{eq:cost_def} is continuous on $X$, and the belief-state penalty $\psi$ in the stage cost~\eqref{eq:cost_def} is lower semicontinuous on $X$.
\end{assumption}

Under this assumption and Assumptions~\ref{ass:prb_set}, the fixed point of $T$ can be identified with the optimal cost function. The following proposition summarizes the main properties of the optimal cost function and optimal  policies.

\begin{proposition}
    \label{prop:bellman_proper_policy}
Let Assumptions~\ref{ass:prb_set} and~\ref{ass:close_U} hold. Then:
\begin{itemize}

    \item [(a)] The optimal cost function $J^*$ defined in~\eqref{eq:j_star_def}
belongs to $\mathcal B_+(X)$ and is the unique solution in $\mathcal B_+(X)$
of the Bellman equation $J=TJ$.
    
    \item [(b)] The optimal cost function $J^*$ is lower semicontinuous.

    \item [(c)] There exists an optimal  policy $\mu^*$, and every 
    optimal  policy is proper, i.e., $\mu^*\in\mathcal M_v$.
    
\end{itemize}
\end{proposition}

Let us describe the main idea behind the proof for Prop.~\ref{prop:bellman_proper_policy}. Our proof analyzes a sequence of functions $\{J_k\}$ generated by VI algorithms; i.e.,
\begin{equation}
    \label{eq:vi}
    J_{k+1}=TJ_k,\qquad k=0,1,\dots.
\end{equation}
For the functions $J_k$, $k=0,1,\dots$, a classical result, first given in \cite{bertsekas1975monotone,bertsekas1977monotone}, states that if the stage cost is nonnegative, $J_0\equiv0$, and there exists some $\bar k$ such that the level set
\begin{equation}
    \label{eq:compact_level_set}
    U_k(x,\lambda)=\Big\{ u\in U\,  \Big| \, g(x,u)+\alpha E\big\{ J_k\big(f(x,u, w)\big)\big\}\le \lambda\Big\}
\end{equation}
is compact for every $k\geq \bar k$, $x\in X$, and $\lambda\in \Re$, then $J_k(x)$ converges to $ J^*(x)$ for all $x\in X$. Moreover, there exists an optimal stationary policy, cf. Prop.~\ref{prop:classical}(b) in~Appendix~\ref{app:classical}. We will verify that the conditions stated in Prop.~\ref{prop:classical}(b) are satisfied in our problem, implying that $J_k$ converges pointwise to $J^*$. On the other hand, Bellman operator $T$ being a contraction implies that $J_k$ converges to the fixed point of $T$ [cf. Prop.~\ref{prop:contraction}], showing that $J^*$ is the fixed point of $T$.

\begin{proof}
    We first show the following lemma, which, among others, verifies the conditions stated in Prop.~\ref{prop:classical}(b) holds in our problem. The proof is long and has been relegated to Appendix~\ref{app:proof_vi_lsc_level_compact}.

    \begin{lemma}[Lower Semicontinuity and Compact Level Sets]
\label{lem:vi_lsc_level_compact}
Suppose Assumptions~\ref{ass:prb_set} and~\ref{ass:close_U} hold. Let $J_0\equiv 0$ and $J_{k+1}=TJ_k$. Then, for every $k\geq 0$, the function $J_k$ belongs to $\mathcal{B}_+(X)$ and is lower semicontinuous. Moreover, for every $k\geq 0$, $x\in X$ and $\lambda\in \Re$, the level set $U_k(x,\lambda)$ defined in Eq.~\eqref{eq:compact_level_set} is compact.
\end{lemma}

Given that Lemma~\ref{lem:vi_lsc_level_compact} hold, we have that the conditions stated in Prop.~\ref{prop:classical}(b) are satisfied. As a result, the sequence $\{J_k\}$ stated in Lemma~\ref{lem:vi_lsc_level_compact} converges pointwise to $J^*$, and there exists a stationary optimal policy.

(a) By Prop.~\ref{prop:contraction}, the Bellman operator $T$ is a contraction on $\mathcal B_+(X)$. As a result, it admits a unique fixed point $\bar J$ in $\mathcal B_+(X)$, and $\|J_k-\bar J\|_v\to0$, where $J_{k+1}=TJ_k$, and $J_0\equiv 0$. This implies that $J_k(x)\to \bar J(x)$ for all $x\in X$. On the other hand, the same sequence $\{J_k\}$ satisfies $J_k(x)\to J^*(x)$ for all $x\in X$. Thus, by uniqueness of pointwise limits, we have $ \bar J = J^*$.

(b) We first note that the sequence $\{J_k\}$ considered in Lemma~\ref{lem:vi_lsc_level_compact} is monotonically increasing in the sense that $J_{k+1}\geq J_k$, $k=0,1,\dots$. Indeed, with the stage cost $g(x,u)$ being nonnegative, we have $J_1=TJ_0\geq J_0$ as $J_0(x)\equiv0$. Applying $k$-fold Bellman operator $T^k$ on both sides and in view of the monotonocity property of $T$, we obtain $J_{k+1}\geq J_k$. Moreover, since $J_k(x)\to J^*$ for all $x$, we also have $J_k\leq J^*$ for all $k$.

Let $\{x_\ell\}$ be a sequence of states that converges to $x$, where $x$ is an arbitrary state. We have shown in Lemma~\ref{lem:vi_lsc_level_compact} that $J_k$ is lower semicontinuous for all $k$. Therefore,
$$J_k(x)\leq \liminf_{\ell\to\infty}J_k(x_\ell)\leq \liminf_{\ell\to\infty}J^*(x_\ell),\qquad k=0,1,\dots,$$
where the last inequality is due to that $J_k\leq J^*$. Taking limit over $k$ on the left side we obtain
$$J^*(x)\leq \liminf_{\ell\to\infty}J^*(x_\ell).$$
Since $x$ is arbitrary, we have that $J^*$ is lower semicontinuous.

(c) In view of Lemma~\ref{lem:vi_lsc_level_compact} and Prop.~\ref{prop:classical}(b), there exists an optimal stationary policy. Let $\mu^*$ be an arbitrary optimal stationary policy, i.e., $J_{\mu^*}=J^*$. Since $J^*\in\mathcal B_+(X)$, we have $J_{\mu^*}\in\mathcal B_+(X)$. By Proposition~\ref{prop:proper_value_equivalence}, it follows that $\mu^*\in\mathcal M_v$. Therefore, we have $\mu^*$ being proper.

\end{proof}

\section{Value and Policy Iteration Algorithms}\label{sec:vi_algo}
In this section, we develop the convergence analysis of three standard DP algorithms within the weighted sup-norm space framework established above: VI, classical PI, and optimistic PI. We show that the functions computed via VI converge geometrically to the optimal cost function from any initial function in $\mathcal B_+(X)$. For PI and optimistic PI, we further show that the policies generated by the algorithms remain proper and that the corresponding cost-function sequences converge to the optimal cost in the weighted sup-norm. 

Note that the specific structure of the belief-state dynamics is used in Section~\ref{sec:weight} to establish the weighted sup-norm contraction properties of the associated Bellman operators. The convergence analysis developed in this section relies primarily on these contraction properties and monotonicity, rather than the particular belief-state dynamics of the present problem. As a result, the vanishing-error convergence results for VI, PI, and optimistic PI extend to the general contractive DP models studied in \cite[Chapter~2]{bertsekas2022abstract}, where the Bellman operator and all policy-specific Bellman operators are contractions on a common weighted sup-norm space.

\subsection{Value Iteration}
Let us consider the VI algorithm that permits minimization errors. In particular, starting with some function $J_0\in \mathcal B_+(X)$, and it generates a sequence of functions $\{J_k\}$ satisfying 
\begin{equation}
    \label{eq:vi_repeat}
    J_{k+1}\in \mathcal{B}_+(X),\quad\|J_{k+1}-TJ_k\|_v\leq \epsilon_k,\qquad k=0,1,\dots,
\end{equation}
where $\{\epsilon_k\}$ is a sequence of positive scalars. In particular, if $\epsilon_k=0$ for all $k$, we obtain the algorithm  \eqref{eq:vi}. We prove the following result.

\begin{proposition}[Convergence of VI]
\label{pro:vi_arbitrary_init}
Let Assumptions~\ref{ass:prb_set} and~\ref{ass:close_U} hold. If the sequence $\{\epsilon_k\}$ satisfies $\epsilon_k\to0$, then for any $J_0\in \mathcal B_+(X)$, the sequence $\{J_k\}$ computed via the VI algorithm~\eqref{eq:vi_repeat} converges to the optimal cost function $J^*$; i.e.,
$$\lim_{k\to\infty}\|J_k-J^*\|_v=0.$$
Moreover, if $J_{k+1}=TJ_k$ for all $k$, we have
$$\|J_k-J^*\|_v  \le \alpha^k \|J_0 -J^*\|_v,\qquad k=0,1,\dots.$$
\end{proposition}

\begin{proof}
    Since $\{J_k\}$ satisfies the inequality \eqref{eq:vi_repeat}, we have
    $$ \|J_{k}-T^2J_{k-2}\|_v\leq \|J_{k}-TJ_{k-1}\|_v+\|TJ_{k-1}-T^2J_{k-2}\|_v\leq \epsilon_{k-1}+\alpha \epsilon_{k-2}.$$
    Proceeding similarly, we obtain
    $$\|J_{k}-T^kJ_{0}\|_v\leq \eta_{k-1},$$
    where $\eta_{k+1}=\alpha\eta_{k}+\epsilon_{k+1}$, $k=0,1,\dots$, and $\eta_0=\epsilon_0$. Taking limit superior on both sides of the equality $\eta_{k+1}=\alpha\eta_{k}+\epsilon_{k+1}$ and since $\epsilon_k\to0$, we obtain that $\|J_{k}-T^kJ_{0}\|_v\to 0$. We have shown in Section~\ref{sec:contraction} that the Bellman operator $T$ is a self-map and a contraction on the weighted sup-norm space $\mathcal B_+(X)$ with modulus $\alpha$. By applying a variant of the Banach fixed point theorem [cf. Prop.~\ref{prop:banach}], we have $\|J^*-T^kJ_{0}\|_v\to 0$, which implies $\|J_k-J^*\|_v\to0$. 
    
    If $J_{k+1}=TJ_k$ for all $k$, the inequality $\|J_k-J^*\|_v  \le \alpha^k \|J_0 -J^*\|_v$ follows directly from Prop.~\ref{prop:banach}.
\end{proof}

\subsection{Classical Policy Iteration}
We now discuss PI within the proper-policy class $\mathcal M_v$.  Starting from a proper policy $\mu^0\in \mathcal M_v$, PI generates a sequence of  policies $\{\mu^k\}$ through a sequence of policy evaluations to obtain $J_{\mu^k}$ from the equation
\begin{equation}
    \label{eq:pi_pe}
    J_{\mu^k}=T_{\mu^k}J_{\mu^k}.
\end{equation}
We then perform the policy improvement step by computing $\mu^{k+1}$ such that 
\begin{equation}
    \label{eq:pi_pi}
    T_{\mu^{k+1}}J_{\mu^{k}}\leq J_{\mu^k},\qquad T_{\mu^{k+1}}J_{\mu^{k}}\leq TJ_{\mu^{k}}+\epsilon_kv,
\end{equation}
where $\{\epsilon_k\}$ is a sequence of positive scalars, which are design parameters. Note that if $T_{\mu^{k+1}}J_{\mu^{k}}= TJ_{\mu^{k}}$, we obtain the classical PI algorithm. We establish the following convergence result.

\begin{proposition}[Convergence of PI]
Suppose Assumptions~\ref{ass:prb_set} and \ref{ass:close_U} hold. Let the initial policy $\mu^0$ be proper. Then the PI algorithm \eqref{eq:pi_pe}-\eqref{eq:pi_pi} is well-posed in the sense that the policies $\mu^k$, $k=0,1,\dots$, are proper. Moreover, if the sequence $\{\epsilon_k\}$ satisfies $\epsilon_k\to0$, then we have 
$$
\lim_{k\to\infty}\|J_{\mu^k} - J^*\|_v = 0.
$$
\end{proposition}

\begin{proof}
We first show that the sequence of policies $\mu^k$ are proper. We have that $\mu^0$ is proper by assumption. Suppose that $\mu^k$ is proper. Then Prop.~\ref{prop:proper_value_equivalence} implies that $J_{\mu^k}\in \mathcal B_+(X)$. Since $\mu^{k+1}$ satisfies $T_{\mu^{k+1}}J_{\mu^{k}}\leq J_{\mu^k}$, we obtain 
$$0\leq g\big(x,\mu^{k+1}(x)\big)\leq (T_{\mu^{k+1}}J_{\mu^{k}})(x)\leq J_{\mu^k}(x),\qquad \forall x\in X,$$
where the second inequality is due to that $J_{\mu^{k}}\in \mathcal B_+(X)$. As a result, we have $\mu^{k+1}$ is proper.

In what follows, we will show that 
\begin{equation}
    \label{eq:pi_bound_1}
    J_{\mu^{k+1}}\leq TJ_{\mu^{k}}+\epsilon_kv,\qquad k=0,1,\dots.
\end{equation}
Since $T_{\mu^{k+1}}J_{\mu^{k}}\leq J_{\mu^k}$, applying on both sides the operator $T_{\mu^{k+1}}$ and using its monotonicity property, we have $T_{\mu^{k+1}}(T_{\mu^{k+1}}J_{\mu^{k}})\leq T_{\mu^{k+1}}J_{\mu^{k}}$. Using Prop.~\ref{prop:classical}(a) with $\mu=\mu^{k+1}$ and $\hat J=T_{\mu^{k+1}}J_{\mu^{k}}$, we obtain $J_{\mu^{k+1}}\leq T_{\mu^{k+1}}J_{\mu^{k}}$. Combining this inequality with the second inequality in step \eqref{eq:pi_pi}, we obtain the desired relation \eqref{eq:pi_bound_1}.

Since the inequality \eqref{eq:pi_bound_1} holds for all $k$, we have
\begin{align*}
    J_{\mu^{k}}\leq& TJ_{\mu^{k-1}}+\epsilon_{k-1}v\leq T(TJ_{\mu^{k-2}}+\epsilon_{k-2}v)+\epsilon_{k-1}v\\
    \leq &T^2J_{\mu^{k-2}}+(\alpha \epsilon_{k-2}+\epsilon_{k-1})v,
\end{align*}
where the second inequality is due to monotonicity of $T$ and the inequality $J_{\mu^{k}}\leq TJ_{\mu^{k-1}}+\epsilon_{k-1}v$; cf. inequality~\eqref{eq:pi_bound_1}. The last inequality is due to the weight shift property of $T$; cf. Eq.~\eqref{eq:weight_shift}. Continue in similar way, we obtain
$$J_{\mu^{k}}\leq T^{k}J_{\mu^{0}}+(\alpha^{k-1}\epsilon_0+\alpha^{k-2}\epsilon_1\cdots+\epsilon_{k-1})v,$$
or equivalently,
\begin{equation}
    \label{eq:pi_bound_2}
    J_{\mu^{k}}\leq T^{k}J_{\mu^{0}}+\eta_{k-1}v,
\end{equation}
where $\{\eta_k\}$ is a sequence of nonnegative scalars defined recursively by $\eta_{k+1}=\alpha\eta_k+\epsilon_{k+1}$, $k=0,1,\dots$, and $\eta_0=\epsilon_0$. Combining the inequality~\eqref{eq:pi_bound_2} with the relation $J^*\leq J_{\mu^{k}}$, we obtain
$$\|J_{\mu^{k}}-J^*\|_v\leq \|T^{k}J_{\mu^{0}}-J^*\|_v+\eta_{k-1}.$$

Finally, by taking limit superior on both sides of the equation $\eta_{k+1}=\alpha\eta_k+\epsilon_k$ and using the fact that $\epsilon_k\to0$, we obtain $\eta_k\to0$. Together with the preceding inequality and the fact that $\|T^{k-1}J_{\mu^{0}}-J^*\|_v\to0$, we obtain the desired convergence $\|J_{\mu^k} - J^*\|_v\to0$.
\end{proof}

\subsection{Optimistic Policy Iteration}
Let us consider a variant of PI, known as optimistic PI, where the policies are evaluated approximately. The algorithm starts with some $J_0\in \mathcal B_+(X)$, and generates a sequence of function-policy pairs $\{(J_k, \mu^k)\}$ as follows. Given $J_k$, we compute a policy $\mu^{k}$ such that
\begin{equation}
    \label{eq:opi_pi}
    T_{\mu^k}J_k\leq TJ_k+\epsilon_kv,
\end{equation}
where $\{\epsilon_k\}$ is a sequence of positive scalars that we can select. We then apply $\ell_k\geq1$ iterations of operator $T_{\mu^k}$ on $J_k$, i.e.,
\begin{equation}
\label{eq:opi_pe}
    J_{k+1}=T_{\mu^{k}}^{\ell_k}J_k.
\end{equation}
Thus, $J_{k+1}$ is a finite-iteration approximation of the cost function
$J_{\mu^k}$ associated with policy $\mu^k$. Equivalently, we approximate $J_{\mu^k}(x_0)$ via
$$J_{k+1}(x_0)
    =
    E\Bigg\{
        \sum_{t=0}^{\ell_k-1}
        \alpha^t g\big(x_t,\mu^k(x_t)\big)
        +
        \alpha^{\ell_k}J_k(x_{\ell_k})
    \Bigg\},$$
where $\{x_t\}$ is the sequence of states generated using the policy $\mu^{k}$ via state equation $x_{t+1}=f\big(x_t, \mu^{k}(x_t), w_t\big)$. We have the following convergence result. 

\begin{proposition}[Convergence of Optimistic PI]
Suppose Assumptions~\ref{ass:prb_set} and \ref{ass:close_U} hold. Let the initial function $J_0$ belong to $\mathcal B_+(X)$. Then the optimistic PI algorithm \eqref{eq:opi_pi}-\eqref{eq:opi_pe} is well-posed in the sense that the policies $\mu^k$, $k=0,1,\dots$, are proper. Moreover, if the sequence $\{\epsilon_k\}$ satisfies $\epsilon_k\to0$, then we have 
$$
\lim_{k\to\infty}\|J_k - J^*\|_v = 0,\qquad \lim_{k\to\infty}\|J_{\mu^k} - J^*\|_v = 0.
$$
\end{proposition}

\begin{proof}
    We first show by induction that the policies $\mu^k$, $k=0,1,\dots$, computed by the  optimistic PI algorithm \eqref{eq:opi_pi}-\eqref{eq:opi_pe} are proper, and the functions $J_k$, $k=0,1,\dots$, belong to $\mathcal B_+(X)$. As the initial step, we have $J_0\in \mathcal B_+(X)$ by assumption. Given that $\mu^0$ satisfies inequality \eqref{eq:opi_pi}, we obtain
    $$0\leq g\big(x,\mu^{0}(x)\big)\leq (TJ_{0})(x)+\epsilon_0v(x),\qquad \forall x\in X.$$
    Since $TJ_0,\epsilon_0v\in \mathcal B_+(X)$, we have that $\mu^0$ is proper. By Prop.~\ref{prop:self_map}(a), $T_{\mu^0}$ is a self map on $\mathcal B_+(X)$, which implies that $J_1\in \mathcal B_+(X)$. Suppose that $J_k\in \mathcal B_+(X)$. By repeating the preceding arguments with $k$ and $k+1$ in places of $0$ and $1$, respectively, we have that $\mu^k$ is proper and $J_{k+1}\in \mathcal B_+(X)$. 

    Next we develop lower bounds for the functions $J_k$, $k=0,1,\dots$. The lower bounds are given by the following inequalities
    \begin{equation}
        \label{eq:opi_lb}
        J^*-\frac{\delta_k}{1-\alpha}v\leq J_k,\qquad k=0,1,\dots,
    \end{equation}
    where $\{\delta_k\}$ is a sequence of nonnegative scalars to be introduced shortly. For this purpose, it is sufficient to prove
    \begin{equation}
        \label{eq:opi_tj_j}
        TJ_k\leq J_k+\delta_kv,\qquad k=0,1,\dots.
    \end{equation}
    Indeed, using Prop.~\ref{prop:monotone_contr_ineq} with $T$, $J_k$, and $\delta_k$ in places of $F$, $J$, and $c$, we have that inequality~\eqref{eq:opi_tj_j} implies the relation \eqref{eq:opi_lb}.
    
    We now proceed to show the inequalities~\eqref{eq:opi_tj_j} by induction. We set $\delta_0=\|TJ_0-J_0\|_v$. This value is well-defined as $J_0$ belongs to $\mathcal B_+(X)$ and $T$ is a self-map on $\mathcal B_+(X)$. As a result, we obtain $TJ_0\leq J_0+\delta_0v$. Suppose that $TJ_{k}\leq J_k+\delta_kv$. Since $T_{\mu^k}J_k\leq TJ_k+\epsilon_kv$, we obtain 
    \begin{equation}
        \label{eq:opi_tmuj_j}
        T_{\mu^k}J_k\leq J_k+\xi_kv,
    \end{equation}
    where $\xi_k=\delta_k+\epsilon_k$. Applying on both sides of the preceding inequality with $T_{\mu^k}^{\ell_k}$, we obtain
    $$T_{\mu^k}^{\ell_k+1}J_k\leq T_{\mu^k}^{\ell_k}(J_k+\xi_kv)\leq T_{\mu^k}^{\ell_k}J_k+\alpha^{\ell_k}\xi_kv=J_{k+1}+\delta_{k+1}v,$$
    where $\delta_{k+1}=\alpha^{\ell_k}\xi_k$. In addition, we also have
    $$T_{\mu^k}^{\ell_k+1}J_k=T_{\mu^k}(T_{\mu^k}^{\ell_k}J_k)=T_{\mu^k}J_{k+1}\geq TJ_{k+1}.$$
    Combining the inequalities $TJ_{k+1}\leq T_{\mu^k}^{\ell_k+1}J_k$ and $T_{\mu^k}^{\ell_k+1}J_k\leq J_{k+1}+\delta_{k+1}v$, we obtain the desired relation.

    In what follows, we construct upper bounds for the functions $J_k$, $k=0,1,\dots$. We have shown that the inequality \eqref{eq:opi_tmuj_j} holds for all $k$. Repeated applying the operator $T_{\mu^k}$ on both sides, and using the weight-shift property, we obtain
    $$T_{\mu^k}^{i+1}J_k\leq T_{\mu^k}^{i}J_k+\alpha^i\xi_kv,\qquad i=1,\dots,\ell_k-1.$$
    Adding over these inequalities, we obtain 
    $$T_{\mu^k}^{\ell_k}J_k\leq T_{\mu^k}J_k+\frac{\alpha-\alpha^{\ell_k}}{1-\alpha}\xi_kv\leq T_{\mu^k}J_k+\frac{\alpha}{1-\alpha}\xi_kv.$$
    Combining the preceding inequality with the relations $J_{k+1}=T_{\mu^k}^{\ell_k}J_k$ and $T_{\mu^k}J_k\leq TJ_k+\epsilon_kv$, we obtain 
    \begin{equation}
        \label{eq:opi_t_tj}
        J_{k+1}\leq TJ_k+\zeta_kv,
    \end{equation}
    where $\zeta_k=\epsilon_k+(\alpha\xi_k)/(1-\alpha)$. Since the inequality~\eqref{eq:opi_t_tj} holds for all $k$, we obtain
    $$J_k\leq TJ_{k-1}+\zeta_{k-1}v\leq T(TJ_{k-2}+\zeta_{k-2}v)+\zeta_{k-1}v\leq T^2J_{k-2}+(\alpha\zeta_{k-2}+\zeta_{k-1})v.$$
    Proceeding similarly leads to
    $$J_k\leq T^kJ_0+(\alpha^{k-1}\zeta_0+\alpha^{k-2}\zeta_1\cdots+\zeta_{k-1})v,$$
    or equivalently,
    \begin{equation}
        \label{eq:opi_ub}
        J_k\leq T^kJ_0+\eta_{k-1}v,
    \end{equation}
    where $\{\eta_k\}$ is a sequence of nonnegative scalars defined recursively by $\eta_{k+1}=\alpha\eta_k+\zeta_{k+1}$, $k=0,1,\dots$, and $\eta_0=\zeta_0$. The inequality~\eqref{eq:opi_ub} provides an upper bound for $J_k$. 
    
    To see that $\|J_k-J^*\|_v\to0$, we combine inequalities~\eqref{eq:opi_lb} and \eqref{eq:opi_ub} and obtain
    $$-\frac{\delta_k}{1-\alpha}v\leq J_k-J^*\leq T^kJ_0-J^*+\eta_{k-1}v.$$
    Since $\|T^kJ_0-J^*\|_v\to0$, it suffices to show that $\delta_k\to0$ and $\eta_k\to0$. Recall that $\delta_0=\|TJ_0-J_0\|_v$ and $\delta_{k+1}=\alpha^{\ell_k}(\delta_k+\epsilon_k)\leq \alpha(\delta_k+\epsilon_k)$. Taking limit superior on both sides of the inequality, and in view that $\epsilon_k\to0$, we obtain $\delta_k\to0$. Since $\xi_k=\delta_k+\epsilon_k$, we also obtain $\xi_k\to0$. As a result, $\zeta_k\to0$, as $\zeta_k=\epsilon_k+(\alpha\xi_k)/(1-\alpha)$. Finally, taking limit superior on both sides of the equality $\eta_{k+1}=\alpha\eta_k+\zeta_{k+1}$, and since $\zeta_k\to0$, we obtain $\eta_k\to0$.  

    Finally, we show that $\|J_{\mu^k} - J^*\|_v \to 0$. Since $T_{\mu^k}J_k\leq J_k+\xi_kv$ [cf. inequality~\eqref{eq:opi_tmuj_j}], applying Prop.~\ref{prop:monotone_contr_ineq} with $T_{\mu^k}$, $J_k$, and $\xi_k$ in places of $F$, $J$, and $c$, we obtain $J_{\mu^k}\leq J_k+\frac{\xi_k}{1-\alpha}v$. Subtracting on both sides $J^*$ and due to the inequality $J^*\leq J_{\mu^k}$, we have
    $$0\leq J_{\mu^k}-J^*\leq J_k-J^*+\frac{\xi_k}{1-\alpha}v.$$
    This implies $\|J_{\mu^k}-J^*\|_v\leq \|J_k-J^*\|_v+\frac{\xi_k}{1-\alpha}$. We have shown that $\|J_k-J^*\|_v\to0$ and $\xi_k\to0$. As a result, $\|J_{\mu^k}-J^*\|_v\to0$. 
\end{proof}

\section{Numerical Example}
\label{sec:numerical_result}

We illustrate the developed DP theory using a scalar example of~\eqref{eq:io} with noise variance $\sigma^2=1$. We use a two-point innovation model with support $W=\{-1,1\}$, where each value occurs with probability $0.5$. This model preserves the zero mean and unit variance of the Gaussian innovation. Consider the following scalar stage cost
$$
g(x_k,u_k)
=
\big(u_k-m_k\big)^2
+
m_k^2
+
2P_k .
$$
The first term penalizes deviations from the reference input $u_k=m_k$, while $m_k^2+2P_k$ penalizes the belief state.
The optimal cost function is computed by VI\footnote{PI and optimistic PI give similar results.} in~\eqref{eq:vi} with discount factor $\alpha=0.5$. 

In the numerical study, the DP policy is approximated by grid-based VI. We discretize the belief-state space and the control set, evaluate the Bellman objective over the control grid at each belief-state grid point, and select the minimizing control by exhaustive grid search.
Specifically, we use a $121\times81$ belief-state grid: the posterior mean is discretized uniformly over $[-6,6]$ with spacing $0.1$, while the posterior variance grid consists of $P=0$ and 80 logarithmically spaced points between $10^{-3}$ and $10$. For clarity, the figures display only the region $(m,P)\in[-3,3]\times[0,3]$. The control grid contains 401 uniformly spaced points over $[-8,8]$, with spacing $0.04$.
Off-grid cost values are evaluated by bilinear interpolation.
VI is initialized with $J_0\equiv0$ and terminated when the maximum absolute difference between successive value iterates, normalized by the weight $v(x)=1+m^2+P$ at each belief-state grid point, falls below $10^{-6}$, or when 500 iterations have
been completed. This procedure yields numerical approximations of the optimal policy $\mu^*$ and the optimal cost function $J^*$ for the two-point innovation model.

For comparison, we consider the reference control law
$$
\nu(m,P)=m.
$$
This policy represents the exploitation input: it tracks the current posterior mean and does not intentionally increase future information. The deviation of the optimal policy $\mu^*(m,P)$ from the reference control law $\nu(m,P)$ therefore captures the exploratory part of the input induced by the DP solution at state $(m,P)$.

Figure~\ref{fig:scalar_value} shows the computed optimal cost function $J^*(m,P)$, while Fig.~\ref{fig:scalar_policy} compares the optimal policy $\mu^*(m,P)$ with the reference input law $\nu(m,P)$. 
We partition the $(m,P)$-space into Regions A and B: Region A corresponds to
$\lvert \mu^*(m,P)-\nu(m,P)\rvert\leq 0.1$, whereas Region B corresponds to
$\lvert \mu^*(m,P)-\nu(m,P)\rvert>0.1$.
The numerical results show that $\mu^*$ deviates from $\nu$ when the posterior mean has small magnitude and the posterior variance is large. This is precisely Region B of Fig.~\ref{fig:scalar_policy} where the exploitation input $\nu(m,P)=m$ provides limited information, because its magnitude is small, while uncertainty about the parameter remains large. The optimal policy therefore increases the input magnitude relative to $\nu$ to generate more information. This exploratory deviation incurs a larger immediate tracking cost but reduces future uncertainty through the belief update. As $P$ decreases or the magnitude of $m$ increases in Region A of Fig.~\ref{fig:scalar_policy}, the value of exploration decreases, and the optimal policy $\mu^*$ gradually approaches the reference control law $\nu$.

\begin{figure}[htbp]
\centering
\includegraphics[width=0.75\linewidth]{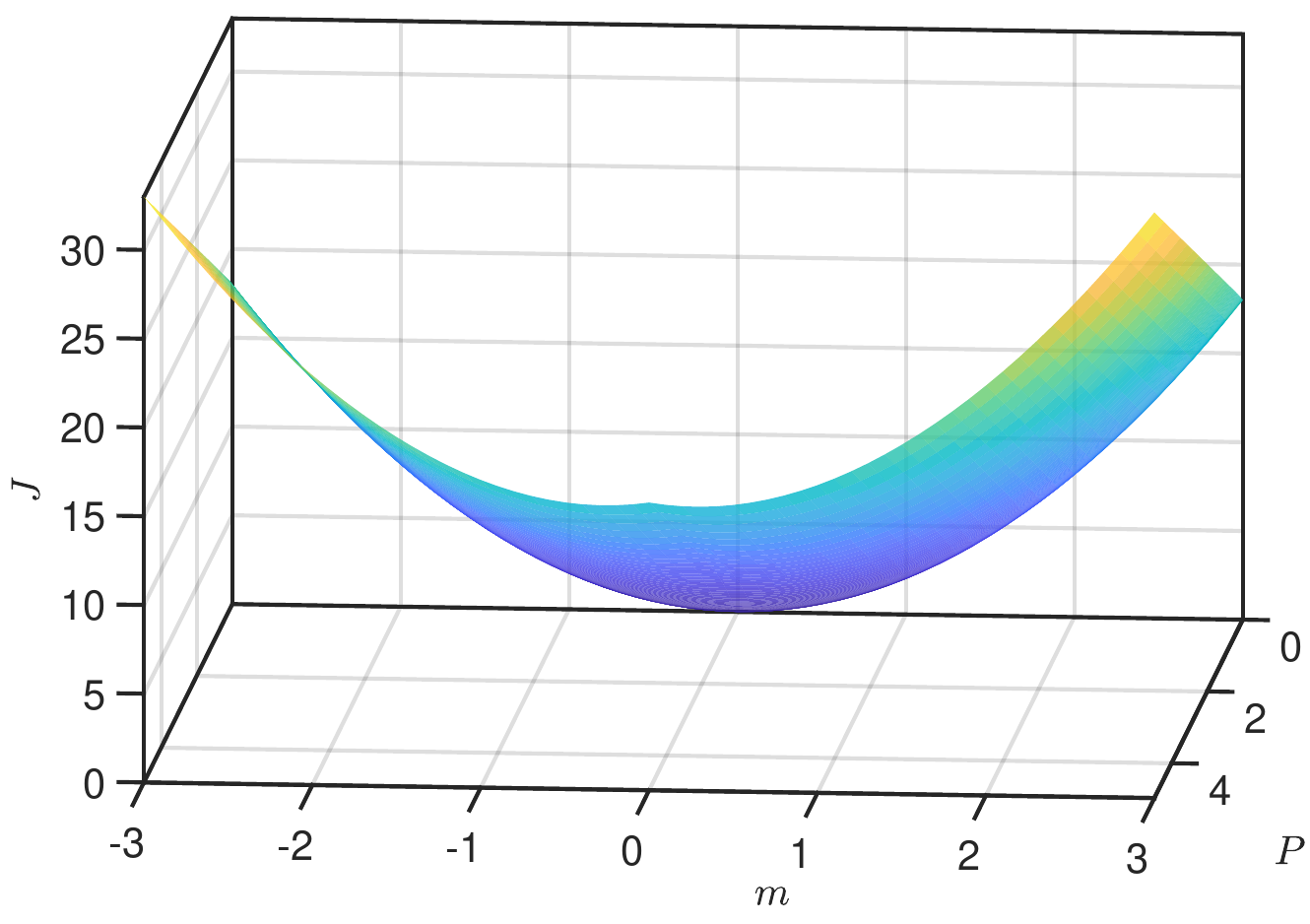}
\caption{Computed optimal cost function $J^*(m,P)$ for the scalar example with $\alpha=0.5$.}
\label{fig:scalar_value}
\end{figure}

\begin{figure}[htbp]
\centering
\includegraphics[width=0.75\linewidth]{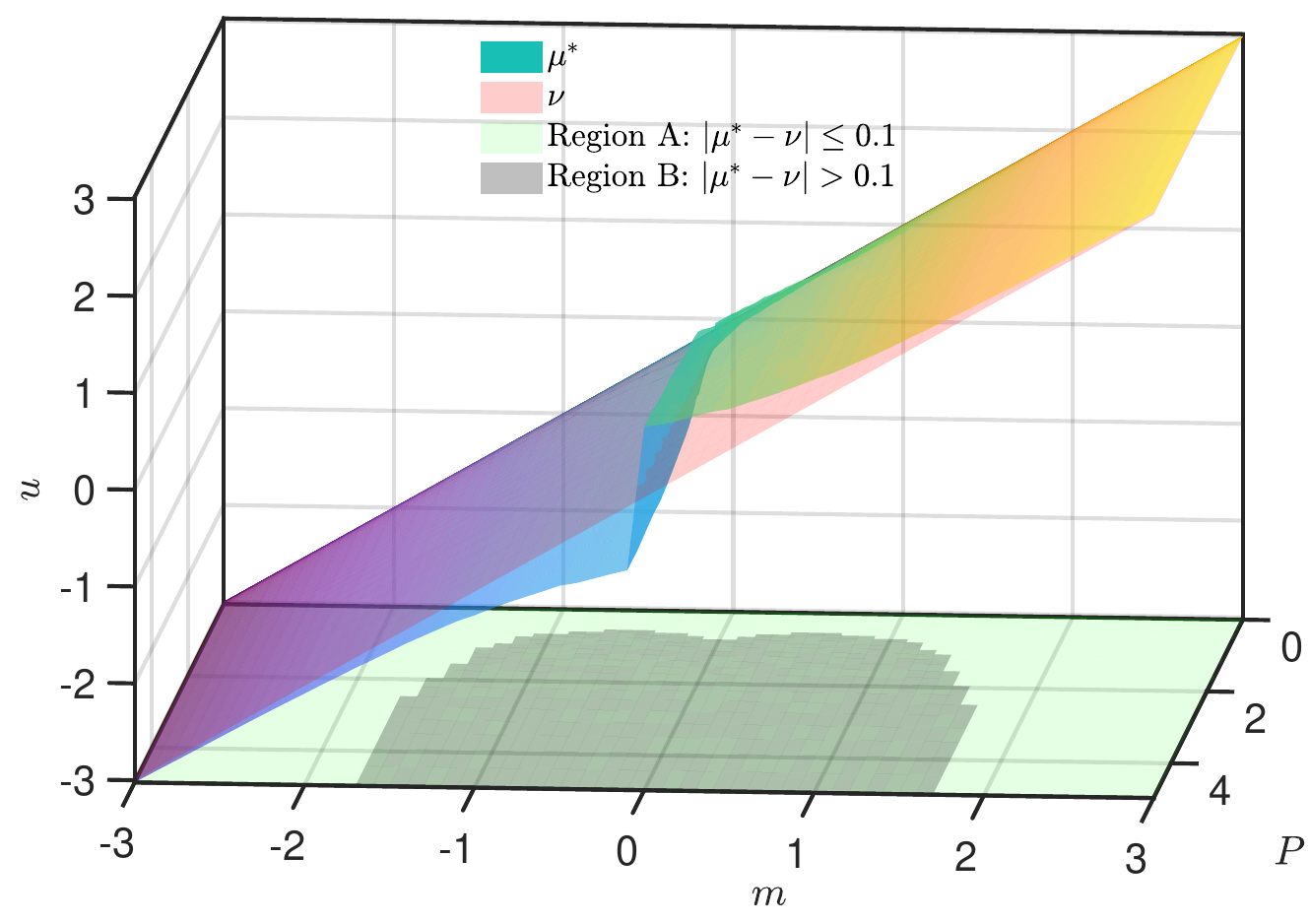}
\caption{Optimal policy $\mu^*(m,P)$ compared with the reference control law $\nu(m,P)=m$.}
\label{fig:scalar_policy}
\end{figure}

\section{Concluding Remarks}\label{sec:conclusion}

In this work, we developed a weighted sup-norm space DP analysis for a countable-innovation surrogate of a linear-Gaussian dual control problem with unbounded belief and control spaces and quadratically growing stage costs. The key structural property is the preservation of the posterior-moment weight under the belief-state transition. This property makes it possible to work directly in a weighted sup-norm space despite the unbounded costs. In particular, it yields self-map and contraction properties of the Bellman operator and of the policy-specific Bellman operators associated with proper policies. Building on these properties, we characterized the optimal cost as the unique solution of the Bellman equation in the weighted space and established the existence and properness of optimal stationary policies. 

We further established weighted sup-norm convergence of VI, PI, and optimistic PI under vanishing approximation errors. For PI and optimistic PI, we showed that the generated policies remain proper; for optimistic PI, both the iterated cost functions and the cost functions of the generated policies converge to the optimal cost. The convergence guarantees for VI, PI, and optimistic PI under vanishing approximation errors are based on general weighted sup-norm contraction arguments and thus extend beyond the specific dual control model studied here.

The notion of a proper policy used here is closely connected with the regular-policy framework of the abstract DP introduced in \cite{bertsekas2017regular}. In that framework, a stationary policy $\mu$ is \emph{regular} with respect to a function space $S$ if its cost function belongs to $S$, satisfies the corresponding Bellman's equation, and the iterates $T_\mu^kJ$ converge to $J_\mu$ for every $J\in S$. In the present problem, taking $S=\mathcal B_+(X)$, this notion coincides with our definition of properness. Indeed, if $\mu$ is proper, then $T_\mu$ is a contraction on $\mathcal B_+(X)$, so that $J_\mu\in\mathcal B_+(X)$ is its unique fixed point and $T_\mu^kJ\to J_\mu$ for all $J\in \mathcal B_+(X)$. Conversely, regularity with respect to $\mathcal B_+(X)$ requires $J_\mu\in\mathcal B_+(X)$, which is equivalent to properness by Prop.~\ref{prop:proper_value_equivalence}. Thus, the proper-policy class $\mathcal M_v$ provides a concrete characterization of the regular policies for the weighted function space used in this paper.

This connection also highlights a distinction between our theory and the broader regular-policy theory. The regular-policy framework is designed to handle semicontractive DP models in which some policies may exhibit pathological fixed-point or iteration behavior and the Bellman operator need not be contractive on the function space of interest. In our problem, the posterior weight-preservation identity yields a genuine contraction structure on $\mathcal B_+(X)$. As a result, the regularity of proper policies follows directly from contraction, while the Bellman operator itself has a unique fixed point in the weighted space. Moreover, under the assumptions used in this paper, every optimal stationary policy is proper. Hence, optimization over the proper-policy class does not exclude optimal stationary policies.

Several directions remain open. The countable-innovation model preserves the first two moments of the Gaussian innovation and hence the posterior-moment identity that underlies the present analysis, but the relation between its optimal cost and policy and those of the original continuous-innovation problem remains to be quantified. Besides, it would be useful to determine how far the weighted sup-norm space and proper-policy arguments developed here can be extended to broader classes of partially observed control problems with unbounded state and control spaces. Some additional DP algorithms, such as $\lambda$-PI discussed in \cite{li2020lambda}, may also worth exploring in the present context.

\appendix

\section{Linear-Gaussian Dual Control Problem}
\subsection{Bayesian Update in Innovation Form}
\label{app:innovation}
Consider the linear-Gaussian model \eqref{eq:io}. Recall that, given the
information $I_k$, the posterior distribution of $\theta$ is Gaussian with
mean $m_k$ and covariance $P_k$. Thus,
\begin{equation*}
    E\{\theta\mid I_k\}=m_k,
    \qquad
    E\big\{(\theta-m_k)(\theta-m_k)'\mid I_k\big\}=P_k.
\end{equation*}
Since $ y_k=u_k'\theta+e_k$, the one-step prediction error can be written as
\begin{equation*}
    y_k-u_k'm_k
    =
    u_k'(\theta-m_k)+e_k.
\end{equation*}
Because $e_k$ is independent of $\theta$ and $I_k$, with zero mean and
variance $\sigma^2$, it follows that
\begin{equation*}
    E\{y_k-u_k'm_k\mid I_k\}=0,\qquad E\big\{(y_k-u_k'm_k)^2\mid I_k\big\}
    =
    \sigma^2+u_k'P_ku_k.
\end{equation*}
Moreover, since $\theta$ and $e_k$ are Gaussian, the prediction error is
Gaussian conditional on $I_k$. Define the standardized innovation by
\begin{equation}
    \epsilon_k=\frac{y_k-u_k'm_k}
    {\sqrt{\sigma^2+u_k'P_ku_k}}.
\label{eq:inno}
\end{equation}
Conditional on $I_k$, the random variable $\epsilon_k$ is Gaussian with zero
mean and unit variance.

We next derive the posterior mean update. 
The conditional cross moment between $\theta-m_k$ and the prediction error is
\begin{align*}
    &E\big\{(\theta-m_k)(y_k-u_k'm_k)\mid I_k\big\} \\
    =&
    E\big\{(\theta-m_k)
    \big(u_k'(\theta-m_k)+e_k\big)\mid I_k\big\} \\
    =&P_ku_k.
\end{align*}
Therefore, by the conditional-mean formula in~\cite[Lem.~2.3]{soderstrom2012discrete} for the random variables
$\theta$ and $y_k$, which are jointly Gaussian conditional on $I_k$,
\begin{align*}
    m_{k+1}
    &=
    m_k
    +
    \frac{P_ku_k}
    {\sigma^2+u_k'P_ku_k}
    \big(y_k-u_k'm_k\big) \\
    &=
    m_k
    +
    \frac{P_ku_k}
    {\sqrt{\sigma^2+u_k'P_ku_k}}
    \epsilon_k,
\end{align*}
which gives~\eqref{eq:m_update}.

Similarly, the conditional-covariance formula for these two jointly Gaussian random variables gives
\begin{align*}
    P_{k+1}
    &=
    P_k
    -
    \frac{
        E\big\{(\theta-m_k)(y_k-u_k'm_k)\mid I_k\big\}
        E\big\{(y_k-u_k'm_k)(\theta-m_k)'\mid I_k\big\}
    }{
        E\big\{(y_k-u_k'm_k)^2\mid I_k\big\}
    } \\
    &=
    P_k
    -
    \frac{P_ku_ku_k'P_k}
    {\sigma^2+u_k'P_ku_k},
\end{align*}
which gives~\eqref{eq:P_update}. These are the standard Bayesian updates for
the linear-Gaussian regression model; see~\cite[Sec.~11]{ljung1999system}.

\subsection{Quadratic Stage Cost}
\label{app:cost_motivation}

We now show how the quadratic stage cost in Eq.~\eqref{eq:cost_def} arises from a natural complete-information stage cost under incomplete information. Suppose that, if the parameter $\theta\in\Re^n$ were known, the complete-information problem would prescribe the affine input
\begin{equation}
    \upsilon^*(\theta)=K\theta+b,
    \label{eq:oracle_policy}
\end{equation}
where $K$ is a nonzero $n$ by $n$ matrix and $b\in\Re^n$. Consider the complete-information stage cost
\begin{equation*}
    \hat g(\theta,u)
    =
    \bigl(u-\upsilon^*(\theta)\bigr)'R
    \bigl(u-\upsilon^*(\theta)\bigr)
    +
    \hat\psi(\theta),
\end{equation*}
where $R$ is positive definite and $\hat\psi:\Re^n\mapsto[0,\infty)$ is independent of the current control. The term $\hat\psi(\theta)$ does not affect the minimizing control, but allows the loss to include an additional penalty that depends on the unknown parameter.

Given the belief state $x=(m,P)$, the posterior mean and covariance of $\theta$ satisfy
\begin{equation*}
    E\{\theta\mid x\}=m,
    \qquad
    E\big\{(\theta-m)(\theta-m)'\mid x\big\}=P.
\end{equation*}
The corresponding posterior expected loss is
\begin{equation}
\begin{aligned}
    g(x,u)
    &=
    E\big\{\hat g(\theta,u)\mid x\big\} \\
    &=
    E\big\{
        (u-K\theta-b)'R(u-K\theta-b)
        \mid x
    \big\}
    +
    E\big\{\hat\psi(\theta)\mid x\big\}.
\end{aligned}
\label{eq:cost_org}
\end{equation}

Replacing the unknown parameter $\theta$ in Eq.~\eqref{eq:oracle_policy} by its posterior mean gives the reference input $\nu(x)=\upsilon^*(m)=Km+b$. Using this definition, we obtain $u-K\theta-b=u-\nu(x)-K(\theta-m)$. Substituting this relation into Eq.~\eqref{eq:cost_org} and expanding the quadratic term gives
\begin{align*}
    &E\big\{
        (u-K\theta-b)'R(u-K\theta-b)
        \mid x
    \big\} \\
    =&
    \bigl(u-\nu(x)\bigr)'R\bigl(u-\nu(x)\bigr)-2\bigl(u-\nu(x)\bigr)'RK
    E\{\theta-m\mid x\}+\\
    &E\big\{
        (\theta-m)'K'RK(\theta-m)
        \mid x
    \big\}.
\end{align*}
Since
\begin{equation*}
    E\{\theta-m\mid x\}=0,\qquad E\big\{
        (\theta-m)'K'RK(\theta-m)
        \mid x
    \big\}
    =
    \operatorname{tr}(K'RKP),
\end{equation*}
we obtain
\begin{equation*}
    g(x,u)
    =
    \bigl(u-\nu(x)\bigr)'R\bigl(u-\nu(x)\bigr)
    +
    \operatorname{tr}(K'RKP)
    +
    E\big\{\hat\psi(\theta)\mid x\big\}.
\end{equation*}
Thus, the posterior expected loss has precisely the form~\eqref{eq:cost_def}, with
\begin{equation}
    \psi(x)
    =
    \operatorname{tr}(K'RKP)
    +
    E\big\{\hat\psi(\theta)\mid x\big\}.
\label{eq:psi_def}
\end{equation}

\subsection{Example~\ref{eg:affine}}
\label{app:eg}

We now recover the particular belief-state penalty used in Example~\ref{eg:affine}. Let the parameter-dependent term in the complete-information stage cost be $\hat\psi(\theta)=\theta'H\theta$, where $H$ is positive definite. Since the posterior mean and covariance of $\theta$ at the belief state $x=(m,P)$ are $m$ and $P$, respectively,
\begin{equation*}
    E\big\{\hat\psi(\theta)\mid x\big\}
    =
    E\big\{\theta'H\theta\mid x\big\}
    =
    m'Hm+\operatorname{tr}(HP).
\end{equation*}
Substituting this relation into Eq.~\eqref{eq:psi_def} gives
\begin{align*}
    \psi(x)
    &=
    \operatorname{tr}(K'RKP)
    +
    m'Hm
    +
    \operatorname{tr}(HP) \\
    &=
    m'Hm
    +
    \operatorname{tr}\bigl((H+K'RK)P\bigr).
\end{align*}
Thus, defining $G=H+K'RK$, we obtain
\begin{equation*}
    \psi(x)
    =
    m'Hm+\operatorname{tr}(GP),
\end{equation*}
which is precisely the belief-state penalty in Example~\ref{eg:affine}.

\section{Infinite Horizon Dynamic Programming and Contraction Mappings}\label{app:classical}
The following proposition summarizes two standard results from infinite horizon DP theory for nonnegative cost problems. They can be found in, for example, \cite[Chapter~4]{bertsekas2012dynamic}.

\begin{proposition}\label{prop:classical}Consider the optimal control problem introduced in Section~\ref{sec:dp_prob_formulation}, where Assumption~\ref{ass:prb_set} may or may not be true. The following statements hold.
\begin{itemize}
    \item[(a)] For every  policy $\mu$, its cost function satisfies the policy-evaluation equation
    $$J_{\mu}(x)=g\big(x,\mu(x)\big)+\alpha E\big\{J_{\mu}\big(f\big(x,\mu(x),w\big)\big)\big\},\qquad \forall x\in X.$$
    Moreover, if a nonnegative function $\hat{J}:X\mapsto[0,\infty]$ satisfies
    \begin{equation}
    \label{eq:classic_mu_bound}
        g\big(x,\mu(x)\big)+\alpha E\big\{\hat{J}\big(f\big(x,\mu(x),w\big)\big)\big\}\leq \hat{J}(x),\qquad \forall x\in X,
    \end{equation}
    then $J_\mu\leq \hat{J}$.
    \item[(b)] Let $\{J_k\}$ be the sequence of functions generated by the VI algorithm \eqref{eq:vi} with $J_0\equiv0$. Suppose that there exists some $\bar k$ such that the level set
    $$U_k(x,\lambda)=\Big\{ u\in U\,  \Big| \, g(x,u)+\alpha E\big\{ J_k\big(f(x,u, w)\big)\big\}\le \lambda\Big\}$$
    is compact for every $k\geq \bar k$, $x\in X$, and $\lambda\in \Re$. Then $J_k(x)\to J^*(x)$ for all $x\in X$. Moreover, there exists an optimal stationary policy.
\end{itemize}
\end{proposition}

Another important result that we need in our analysis is a modified version of the Banach fixed-point theorem, which we state below; cf. \cite[Prop.~B.1]{bertsekas2022abstract}.
 
\begin{proposition}\label{prop:banach}
Let $Z$ be a complete vector space equipped with norm $\|\cdot\|_v$ and let $\overline{Z}$ be a closed subset of $Z$. Then if $F:\overline Z\mapsto\overline Z$ is a contraction mapping with modulus $\rho\in (0,1)$, there exists a unique $z^*\in \overline Z$ such that 
$$z^*=Fz^*.$$
Moreover, for any $z\in \overline{Z}$, the sequence $\{F^kz\}$ converges to $z^*$ in the sense that $\|F^kz-z^*\|_v\to0$, where $F^k$ denotes the $k$-fold composition of $F$, and we have 
$$\|F^kz-z^*\|\leq \rho^k\|z-z^*\|_v,\qquad k=1,2,\dots.$$
\end{proposition}

Finally, we state a result that applies to contraction mappings on $\mathcal{B}_+(X)$ with monotonicity property; cf. \cite[Prop.~2.1.4(b)]{bertsekas2022abstract}.
 \begin{proposition}\label{prop:monotone_contr_ineq}
Let $F$ be a contraction mapping on the closed set $\mathcal{B}_+(X)$ with modulus $\rho\in (0,1)$, and let $\hat J$ be its unique fixed point. Suppose that $F$ is monotone in the sense that for all $J,\tilde J\in \mathcal{B}_+(X)$, $J\leq \tilde J$ implies $FJ\leq F\tilde J$. For any $J\in \mathcal{B}_+(X)$ and $c\geq 0$, we have $FJ\leq J+cv$ implies that $\hat J\leq J+\frac{c}{1-\rho}v$. 
\end{proposition}

\section{Proof of Lemma~\ref{lem:vi_lsc_level_compact}}
\label{app:proof_vi_lsc_level_compact}
We first show that for every $k\geq 0$, the function $J_k$ belongs to $\mathcal{B}_+(X)$ and is lower semicontinuous. We proceed by induction. The claim is true for $J_0\equiv 0$. Suppose that the statement holds for $k$. By Prop.~\ref{prop:self_map}, the operator $T$ is a self map on $\mathcal{B}_+(X)$. Therefore, we have $J_{k+1}$ belongs to $\mathcal{B}_+(X)$.

Define $Q_k:X\times U\mapsto \Re$ by
\begin{equation*}
    Q_k(x,u)
    =
    g(x,u)
    +
    \alpha E\big\{
        J_k\big(f(x,u,w)\big)
    \big\}.
\end{equation*}
In what follows, we show that $Q_k(x,u)$ is lower semicontinuous on $X\times U$.

Let $x$ and $u$ be arbitrary state and control, and $\{(x_\ell,u_\ell)\}$ be a sequence of state-control pairs such that $x_\ell\to x$ and $u_\ell\to u$, where $x_\ell=(m_\ell,P_\ell)$. Since
$f_1(x,u,w)$ and $f_2(P,u)$ are continuous, for every $w\in W$, we have
\begin{equation*}
\bigl(f_1(x_\ell,u_\ell,w),f_2(P_\ell,u_\ell)\bigr)
\to
\bigl(f_1(x,u,w),f_2(P,u)\bigr).
\end{equation*}
Since $J_k$ is lower semicontinuous,
\begin{equation}
    \label{eq:j_k_semicon}
    J_k\bigl(f_1(x,u,w),f_2(P,u)\bigr)\le\liminf_{\ell\to\infty}J_k\bigl(f_1(x_\ell,u_\ell,w),f_2(P_\ell,u_\ell)\bigr),
\end{equation}
Since $J_k\ge 0$, taking expectation on both sides of
\eqref{eq:j_k_semicon} with respect to $w$ gives
\begin{align*}
E\big\{J_k\big(f(x,u,w)\big)\big\}
&\leq 
E\Big\{\liminf_{\ell\to\infty}J_k\big(f(x_\ell,u_\ell,w)\big)\Big\}\\
&\le
\liminf_{\ell\to\infty}
E\big\{J_k(f(x_\ell,u_\ell,w))\big\}.
\end{align*}
where the second inequality is due to Fatou's lemma; see, e.g., \cite[Section VI.11]{doob1993measure}. 
Thus this term is lower semicontinuous. Since $g$ is lower semicontinuous under Assumption~\ref{ass:close_U}, $Q_k$ is lower semicontinuous on $X\times U$.

We next show that, for each fixed $x\in X$ and $\lambda\in \Re$, the level set $U_k(x,\lambda)=
\{u\in U\mid Q_k(x,u)\le \lambda\}$ is compact. Since $J_k\ge 0$, we have 
$$
Q_k(x,u)\ge g(x,u) = \big(u-\nu(x)\big)'R\big(u-\nu(x)\big)+\psi(x).
$$
Because $R$ is positive definite, $R$ has a positive minimum eigenvalue, denoted by $\lambda$, such that 
$$
g(x,u) \ge \lambda \|u-\nu(x)\|^2.
$$
For fixed $x$, $\nu(x)$ is fixed. Therefore, if $\|u\|\to \infty$, we have $g(x,u)\to \infty$, and consequently, $Q_k(x,u)\to\infty$. It follows that $U_k(x,\lambda)$ is bounded. Moreover, since $U$ is closed by Assumption~\ref{ass:close_U}, and $Q_k(x,u)$ is lower semicontinuous, the level set $U_k(x,\lambda)$ is closed. 
Lastly, $U_k(x,\lambda)$ is bounded and closed, and hence compact.

It remains to show that $J_{k+1}=TJ_k$ is lower semicontinuous under Assumption~\ref{ass:close_U}. 
Let $\{x_\ell\}$ a convergent sequence in $X$ with its limit as $x$. We will show that
\begin{equation*}
J_{k+1}(x)\le \liminf_{\ell\to\infty}J_{k+1}(x_\ell).
\end{equation*}
If
\begin{equation*}
\liminf_{\ell\to\infty}J_{k+1}(x_\ell)=+\infty,
\end{equation*}
there is nothing to prove. Hence suppose that
\begin{equation}
    \label{eq:L}
    \liminf_{\ell\to\infty}J_{k+1}(x_\ell)=c_1<+\infty.
\end{equation}
By the definition of the lower limit, there exists a subsequence $\{x_\ell\}_{L_1}$ of $\{x_\ell\}$ such that
\begin{equation}
    \label{eq:L1}
    \lim_{\ell\to\infty,\ell\in L_1}J_{k+1}(x_\ell)= c_1<+\infty,
\end{equation}
which means that $J_{k+1}(x_\ell)$ is bounded above.
For each $\ell\in L_1$, since $Q_k(x_\ell, \cdot)$ is lower semicontinuous and has compact level sets, the infimum
\begin{equation*}
J_{k+1}(x_\ell)=\inf_{u\in U}Q_k(x_\ell,u)
\end{equation*}
is attained at the minimizer, which is denoted by $u_\ell$, such that
\begin{equation*}
J_{k+1}(x_\ell)=Q_k(x_\ell,u_\ell).
\end{equation*}
Since $\{J_{k+1}(x_\ell)\}$ is bounded above, there exists $c_2<+\infty$ such that
\begin{equation*}
Q_k(x_\ell,u_\ell)\le c_2
\end{equation*}
for all sufficiently large  $\ell\in L_1$. Since $J_k\ge 0$, the expectation term in $Q_k$ is nonnegative, and therefore
\begin{equation*}
Q_k(x,u)\ge g(x,u), \qquad \forall (x,u)\in X\times U.
\end{equation*}
Hence, combining the above two inequalities gives
\begin{equation*}
g(x_\ell,u_\ell)\le Q_k(x_\ell,u_\ell)\le c_2
\end{equation*}
for all sufficiently large $\ell\in L_1$.
If $\{u_\ell\}$ were unbounded, the stage cost function~\eqref{eq:cost_def}:
\begin{equation*}
g(x_\ell,u_\ell)\to+\infty,
\end{equation*}
which contradicts the boundedness above. 
Thus $\{u_\ell\}$ is bounded,
and there exists a subsequence of $\{u_\ell\}_{L_1}$, which we denote as $\{u_\ell\}_{L_2}$, and some $\bar u \in U$, such that
\begin{equation*}
\lim_{\ell\to\infty,\ \ell\in L_2}u_\ell=\bar u.
\end{equation*}
Since $u_\ell\in U$ for all $\ell\in L_2$ and $U$ is closed by Assumption~\ref{ass:close_U}, we have $\bar u\in U$.
Furthermore,  due to $x_\ell \to x $, it follows that
\begin{equation*}
 (x_\ell, u_\ell) \to (x, \bar u), \quad {\rm as } \quad \ell\to\infty,\ \ell\in L_2.
\end{equation*}
By the joint lower semicontinuity of $Q_k$  on $X\times U$,
\begin{equation*}
Q_k(x,\bar u)
\le
\liminf_{\ell\to\infty,\ \ell\in L_2}Q_k(x_\ell,u_\ell)=\liminf_{\ell\to\infty,\ \ell\in L_2}
J_{k+1}(x_\ell).
\end{equation*}
Since $\bar u\in U$, it is feasible at state $x$. Therefore,
\begin{equation*}
J_{k+1}(x)
=
\inf_{u\in U}Q_k(x,u)
\le
Q_k(x,\bar u).
\end{equation*}
Combining the last two inequalities gives
\begin{equation}\label{eq:temp_ineq}
J_{k+1}(x)
\le
Q_k(x,\bar u)
\le
\liminf_{\ell\to\infty,\ \ell\in L_2}J_{k+1}(x_\ell).
\end{equation}
By~\eqref{eq:L},~\eqref{eq:L1}, $L_1$ was chosen so that
\begin{equation*}
\lim_{\ell\to\infty,\ell\in L_1}J_{k+1}(x_\ell)= c_1=\liminf_{\ell\to\infty}J_{k+1}(x_\ell),
\end{equation*}
and since $L_2\subset L_1$, we have
\begin{equation*}
\liminf_{\ell\to\infty,\ \ell\in L_2}J_{k+1}(x_\ell)
=
\liminf_{\ell\to\infty}J_{k+1}(x_\ell).
\end{equation*}
Therefore, combining it and~\eqref{eq:temp_ineq}, 
\begin{equation*}
J_{k+1}(x)
\le
\liminf_{\ell\to\infty}J_{k+1}(x_\ell).
\end{equation*}
Hence $J_{k+1}$ is lower semicontinuous.

\section{Supplementary Proofs for Section~\ref{sec:weight}}\label{app:additional_proof}

\subsection{Proof of Prop.~\ref{prop:proper}}
\label{app:proper}
Since $U$ is nonempty, let us select an arbitrary $u_0\in U$, and define a  policy $\mu_c$ as $\mu_c(x) \equiv u_0$. By Assumption~\ref{ass:prb_set}, there exist constants
$c_\nu>0$ and $c_\psi>0$ such that
\begin{equation*}
    \|\nu(x)\|^2\le c_{\nu} v(x), \qquad \psi(x)\le c_\psi v(x),
    \qquad \forall x\in X.
\end{equation*}
Hence, for every $x\in X$,
\begin{equation*}
\begin{aligned}
    g(x,\mu_c(x))=g(x,u_0)
    &=
    \bigl(u_0-\nu(x)\bigr)'R\bigl(u_0-\nu(x)\bigr)
    +
    \psi(x) \\
    &\le
    \lambda \|u_0-\nu(x)\|^2
    +
    c_\psi v(x) \\
    &\le
    2\lambda \|u_0\|^2
    +
    2\lambda \|\nu(x)\|^2
    +
    c_\psi v(x) \\
    &\le 2\lambda  \|u_0\|^2
    +
    \bigl(2\lambda  c_\nu+c_\psi\bigr)v(x) \\
     &\le
    \Bigl(
    2\lambda \|u_0\|^2
    +
    2\lambda c_\nu
    +
    c_\psi
    \Bigr)v(x),
\end{aligned}
\end{equation*}
where $\lambda$ denotes the maximum eigenvalue of $R$. The second inequality follows from the relation
\begin{equation*}
\begin{aligned}
    \|u_0-\nu(x)\|^2
    &\le \bigl(\|u_0\|+\|\nu(x)\|\bigr)^2 \\
    &= \|u_0\|^2 + 2\|u_0\|\,\|\nu(x)\| + \|\nu(x)\|^2 \\
    &\le 2\|u_0\|^2+2\|\nu(x)\|^2,
\end{aligned}
\end{equation*}
where the first inequality comes from the triangle inequality and the second inequality comes from $2\|u_0\|\,\|\nu(x)\| \leq \|u_0\|^2+\|\nu(x)\|^2$.
Therefore, we have $\mu_c\in\mathcal M_v$.   

\subsection{Proof of Prop.~\ref{prop:self_map}}
\label{app:self_map}
(a) 
 Recall that according to the definition of $T_\mu$, we have
$$(T_\mu J)(x)=g\big(x,\mu(x)\big)+\alpha E\Big\{J\Big(f\big(x,\mu(x),w\big)\Big)\Big\}.$$
By Prop.~\ref{prop:weight_preservation}, for $J\in\mathcal{B}_+(X)$,
$$0\le E\Big\{J\Big(f\big(x,\mu(x),w\big)\Big)\Big\}
\le
\|J\|_v\,v(x),
\quad \forall x\in X.$$
Furthermore,  since $\mu\in\mathcal M_v$, there exists a constant $c_{\mu}>0$ such that $g(x,\mu(x))\leq c_{\mu}v(x)$. Thus,
\begin{equation}\label{eq:up_JT_mu}
\begin{aligned}
        (T_{\mu}J)(x)&=g(x,\mu(x))+\alpha E\Big\{J\Big(f\big(x,\mu(x),w\big)\Big)\Big\}
        \\
    &\le c_{\mu} v(x) +\alpha \|J\|_v v(x)\\
    &\le \big( c_{\mu}+\alpha\|J\|_v\big) v(x).
\end{aligned}
\end{equation}
Together with $(T_\mu J)(x)\ge0$, which comes from $g\ge0$ and $J\in\mathcal B_+(X)$,  we have
\begin{equation*}
| (T_{\mu}J)(x) |\leq ( c_{\mu} +\alpha\|J\|_v\bigr)v(x).
\end{equation*}
Since $v(x)\geq 0$ for all $x\in X$, dividing both sides with $v(x)$ and taking the supremum over $x\in X$, we obtain
\begin{equation*}
\|T_{\mu}J\|_v\le  c_{\mu}+\alpha\|J\|_v.
\end{equation*}
Thus $T_{\mu}J\in\mathcal{B}_+(X)$, and $T_\mu$ is a self-map on $\mathcal{B}_+(X)$.

(b) By Prop.~\ref{prop:proper}, the set of the proper policies $\mathcal M_v$ is nonempty. Let $\mu$ be a proper policy. By the definition of Bellman operator $T$, we have 
$$(TJ)(x)=\inf_{u\in U}\Big[g(x,u)+\alpha E\Big\{J\Big(f\big(x,u,w\big)\Big)\Big\} \Big]\leq  (T_\mu J)(x), \qquad \forall x\in X.$$
Combining this inequality with \eqref{eq:up_JT_mu} leads to
$$(TJ)(x)\leq \big(c_\mu+\alpha\|J\|_v\big)v(x),\qquad \forall x\in X.$$
Since $g\ge0$ and $J\in\mathcal B_+(X)$, we also have $TJ\ge 0$, thus we obtain
$$| (TJ)(x) |\leq \big(c_\mu+\alpha\|J\|_v\big)v(x).$$
Dividing both sides with $v(x)$ and taking the supremum over $x\in X$, we obtain the inequality $\|TJ\|_v\le  c_{\mu}+\alpha\|J\|_v$. Thus, $TJ\in\mathcal{B}_+(X)$, and the Bellman operator $T$ is a self-map on $\mathcal{B}_+(X)$.

\subsection{Proof of Prop.~\ref{prop:contraction}}
\label{app:contraction}
Let $J,\tilde J\in \mathcal B_+(X)$ and $\mu\in \mathcal M$. For all $x\in X$, we have
\begin{equation}
    \label{eq:T_mu_diff}
\begin{aligned}
    &\big|(T_\mu J)(x)-(T_\mu \tilde J)(x)\big|\\
    =&\bigg|\alpha E\Big\{J\Big(f\big(x,\mu(x),w\big)\Big)\Big\}-\alpha E\Big\{\tilde J\Big(f\big(x,\mu(x),w\big)\Big)\Big\}\bigg|\\
    =&\bigg|\alpha E\Big\{J\Big(f\big(x,\mu(x),w\big)\Big)-\tilde J\Big(f\big(x,\mu(x),w\big)\Big)\Big\}\bigg|\\
    \leq &\alpha E\Big\{\Big|J\Big(f\big(x,\mu(x),w\big)\Big)-\tilde J\Big(f\big(x,\mu(x),w\big)\Big)\Big|\Big\} 
\end{aligned}    
\end{equation}
where the last inequality is due to Jensen's inequality; see, e.g., \cite[Section VI.11]{doob1993measure}.

Since $J,\tilde J\in \mathcal{B}_+(X)$, for every $y\in X$, we have 
\begin{equation*}
    |J(y)-\tilde J(y)|
    \le
    \|J-\tilde J\|_v v(y).
\end{equation*}
Taking $y=f(x,\mu(x),w)$, we obtain
\begin{equation*}
\Big|J\Big(f\big(x,\mu(x),w\big)\Big)-\tilde J\Big(f\big(x,\mu(x),w\big)\Big)\Big| \le 
\| J-\tilde J\|_v\,v\Big(f\big(x,\mu(x),w\big)\Big).
\end{equation*}
Combining the preceding relation with inequality~\eqref{eq:T_mu_diff} gives
\begin{equation*}
|(T_{\mu} J)(x)-(T_{\mu}\tilde J)(x)|
\le
\alpha 
\| J-\tilde J\|_v\,E\Big\{v\Big(f\big(x,\mu(x),w\big)\Big)\Big\}
\end{equation*}
By Prop.~\ref{prop:weight_preservation}, we have $E\{v(f(x,\mu(x),w))\}=v(x)$. Combining with the preceding relation yields
\begin{equation}\label{eq:contraction_up_bound}
|(T_{\mu}J)(x)-(T_{\mu}\tilde J)(x)|
\le
\alpha \|J-\tilde J\|_v\,v(x).
\end{equation}
Dividing on both sides $v(x)$ and taking supremum over $x\in X$, we obtain
\begin{equation}
\label{eq:contraction_app}
\|T_{\mu}J-T_{\mu}\tilde J\|_v\le \alpha \|J-\tilde J\|_v.
\end{equation}

(a) We have shown that inequality \eqref{eq:contraction_app} holds for all $\mu\in \mathcal{M}$. By Prop.~\ref{prop:self_map}, if $\mu$ is proper, its corresponding operator $T_\mu$ is a self map on $\mathcal B_+(X)$. As a result, it is a contraction on $\mathcal B_+(X)$. On the other hand, if $\mu\not\in \mathcal M_v$, the operator $T_\mu$ needs not be a self-map on $\mathcal B_+(X)$ and therefore is not a contraction.

(b) Let $J,\tilde J\in \mathcal B_+(X)$. For all $\mu\in \mathcal M$ and $x\in X$, we have
\begin{align*}
    (T_\mu J)(x)=&(T_\mu \tilde J)(x)+(T_\mu J)(x)-(T_\mu \tilde J)(x)\\
    \leq &(T_\mu \tilde J)(x)+\big|(T_\mu J)(x)-(T_\mu \tilde J)(x)\big|\\
    \leq &(T_\mu \tilde J)(x)+\alpha \|J-\tilde J\|_v\,v(x),
\end{align*}
where the last inequality is due to the relation \eqref{eq:contraction_up_bound}. Taking infimum over $\mu$ on both sides and in view of the relation
$$(TJ)(x)=\inf_{\mu\in \mathcal M}(T_\mu J)(x),\quad (T\tilde J)(x)=\inf_{\mu\in \mathcal M}(T_\mu \tilde J)(x),\qquad \forall x\in X,$$
we obtain
$$(T J)(x)\leq (T \tilde J)(x)+\alpha \|J-\tilde J\|_v\,v(x).$$
Reversing the roles of $J$ and $\tilde J$, we also have
$$(T \tilde J)(x)\leq (T J)(x)+\alpha \|J-\tilde J\|_v\,v(x).$$
Combining the preceding two relations yields the desired inequality.

\printbibliography

\end{document}